\documentclass{amsart}
\usepackage{amssymb}
\usepackage{amsfonts}
\usepackage{amssymb}
\usepackage{amsmath}
\usepackage{amsthm}
\usepackage{enumerate}
\usepackage{soul}
\usepackage{tabularx}
\usepackage{centernot}
\usepackage{mathtools}
\usepackage{stmaryrd}
\usepackage{amsthm,amssymb}
\usepackage{etoolbox}
\usepackage{tikz}
\usepackage{amssymb}
\usepackage{tablefootnote}
\usetikzlibrary{matrix}
\usepackage{graphics,graphicx}
\usepackage[all]{xy}
\usepackage{tikz-cd}
\usepackage{url}
\usepackage{marginnote}
\definecolor{mygray}{gray}{0.85}
\usepackage[linecolor=black,backgroundcolor=mygray,colorinlistoftodos,prependcaption,textsize=small]{todonotes}
\usepackage{xargs}

\renewcommand{\leq}{\leqslant}
\renewcommand{\geq}{\geqslant}

\newcommand\myrestriction{\mathord\restriction}

\newcommand{\Ascr}{\mathcal A}
\newcommand{\Bscr}{\mathcal B}

\newcommand{\Dscr}{\mathcal D}
\newcommand{\Gscr}{\mathcal G}

\newcommand{\Uscr}{\mathcal U}

\def\abar{\mbox{\boldmath $a$}}
\def\bbar{{\bf b}}

\def\dbar{{\bf d}}
\def\ebar{{\bf e}}

\def\hbar{{\bf h}}

\def\xbar{{\bf x}}

\def\aut{\rm aut}

\def\icl{\rm icl}

\def\cl{\rm cl}
\def\tp{\rm tp}
\def\acl{\rm acl}

\makeatletter
\def\subsection{\@startsection{subsection}{3}%
  \z@{.5\linespacing\@plus.7\linespacing}{.3\linespacing}%
  {\bfseries\centering}}
\makeatother

\makeatletter
\def\subsubsection{\@startsection{subsubsection}{3}%
  \z@{.5\linespacing\@plus.7\linespacing}{.3\linespacing}%
  {\centering}}
\makeatother

\makeatletter
\def\myfnt{\ifx\protect\@typeset@protect\expandafter\footnote\else\expandafter\@gobble\fi}
\makeatother

\newtheorem{theorem}{Theorem}[section]

\newtheorem{corollary}[theorem]{Corollary}
\newtheorem{axiom}[theorem]{Axiom}
\newtheorem{definition}[theorem]{Definition}
\newtheorem{lemma}[theorem]{Lemma}

\newtheorem{example}[theorem]{Example}

\newtheorem{question}[theorem]{QUESTION}

\newtheorem{fact}[theorem]{Fact}

\newtheorem{remark}[theorem]{Remark}
\newtheorem{notation}[theorem]{Notation}

\newtheorem{assumption}[theorem]{Assumption}

\newtheorem{definition/fact}[theorem]{Definition/Fact}

\def\bU{\mbox{\boldmath $U$}}
\def\bF{\mbox{\boldmath $F$}}
\def\bF{\mbox{\boldmath $F$}}
\def\bJ{\mbox{\boldmath $J$}}
\def\bK{\mbox{\boldmath $K$}}
\def\bL{\mbox{\boldmath $L$}}
\def\bbL{\mbox{\boldmath $\hat{L}$}}
\def\b1L{\mbox{\boldmath ${L}_{-1}$}}
\def\Fscr{\mathcal F}

\newcommand{\pureindep}[1][]{%
  \mathrel{
    \mathop{
      \vcenter{
        \hbox{\oalign{\noalign{\kern-.3ex}\hfil$\vert$\hfil\cr
              \noalign{\kern-.7ex}
              $\smile$\cr\noalign{\kern-.3ex}}}
      }
    }\displaylimits_{#1}
  }
}

\newcommand{\indep}[2]{%
  \mathrel{
    \mathop{
      \vcenter{
        \hbox{%
\oalign{
\noalign{\kern-.3ex}\hfil$\vert$\hfil\cr
              \noalign{\kern-.7ex}
              $\smile$\cr\noalign{\kern-.3ex}
}
}
      }
}^{\!\!\!\!\!#2}_{\!\!\hspace{-0.1em}#1}
  }
}

\date{\today}
\begin{document}
\title{Strongly minimal Steiner Systems III: Path graphs and sparse configurations}

\author{John T. Baldwin
\\University of Illinois at Chicago\\
}
\thanks{Research partially supported by Simons
travel grant  G3535.}

\renewcommand{\shorttitle}{Steiner Systems: Path graphs and Sparse Configurations}

\begin{abstract} We introduce  a uniform method
of proof for the following results.   For {\em each} of the following
conditions, there are $2^{\aleph_0}$ families of Steiner systems,
satisfying that condition:  i) Theorem~\ref{getsparse}: (extending
\cite{Chicoetal}) each Steiner triple system is $\infty$-sparse and has a
uniform but not perfect path graph; ii) (Theorem~\ref{restrictk0}:
(extending \cite{CameronWebb}) each Steiner $k$-system (for $k=p^n$) is
$2$-transitive and has a uniform path graph (infinite cycles only); iii)
Theorem~\ref{omitmitre}: (extending \cite{Fujiwaramitre}, each is
anti-Pasch (anti-mitre); iv) Theorem~\ref{getsmquasigrp} has an explicit
quasi-group structure. In each case  all members of the family satisfy the
same complete strongly minimal theory and it has $\aleph_0$ countable
models and one model of each uncountable cardinal.
\end{abstract}

%
%

\maketitle

We extend various properties of finite Steiner systems to infinite systems.
We take from model theory  the practice of constructing families of
structures (models of a particular theory) satisfying the goal property and
take properties of interest from both model theory and combinatorics. We are
able to both obtain examples of known phenomena of infinite Steiner systems
that satisfy additional very strong model theoretic constraints and adapt
certain concepts from Steiner triple systems to Steiner $q$-systems for prime
power $q$.

We extend recent work by Barbina and Casanovas in model theory
\cite{BarbinaCasa} and by Horsley and Webb in
 combinatorics \cite{HorsleyWebb}
 that centers on the
 constructions of Steiner triple systems (Section~\ref{homconfus})   by
giving some applications  of the Hrushovski construction of strongly minimal
sets. Our subject  differs from the Fra{\" \i}ss\'{e} case because the finite
structures must be `strong' substructures of the generic (alias, limit) model
and we can vary the meaning of strong as discussed in Notation~\ref{hruclass}
and Section~\ref{homconfus}.   In contrast to the model theoretically complex
locally finite generics that arise from Fra{\" \i}ss\'{e} construction
\cite{BarbinaCasa}, the construction techniques here give theories that are
strongly minimal, the geometric building blocks of model theoretically tame
structures.

More significantly from a combinatorial standpoint, we systematically extend
these model theoretic methods to the study  of infinite Steiner $p^k$-systems
with such properties considered for finite Steiner {\em triple} systems (STS)
as anti-Pasch, sparseness, quasigroup structure, and cycle graphs.
Section~\ref{bg} sketches background information on the Hrushovski
construction and its relation to other generalization of the
 Fra{\" \i}ss\'{e} method. We reformulate (Section~\ref{omcon}) the
Cameron-Webb notion of sparse configurations \cite{Chicoetal,Fujiwaramitre}
in terms of the $\delta$-function fundamental to the Hrushovski construction
and give uniform accounts of the existence  in every infinite cardinality of
anti-Pasch, anti-mitre and indeed $\infty$-sparse Steiner triple systems
 (Corollaries \ref{omitPasch} and \ref{omitmitre} and
Theorem~\ref{getsparse}).
 While the examples of strongly
minimal pure Steiner $k$-systems, $(M,R)$,  with $k >3$ admit no definable
`truly binary' operation with infinite domain
\cite{BaldwinVer,BaldwinsmssII}, we construct (Theorem~\ref{getsmquasigrp})
strongly minimal quasigroups which induce $q$-Steiner systems (line length
$q$)  for $q$ a prime power. This is our first extension from $3$ to $p^n$.
In contrast to \cite{HorsleyWebb} rather than omitting appropriate sets $\bF$
of finite systems, we require the quasigroup determining the Steiner system
to be in a fixed variety of quasigroups \cite{GanWer2}.

Our second extension (Section~\ref{smba}) from $3$ to $p^n$  moves the notion
of an $(a,b)$-cycle graph $G_M(a,b)$ of an infinite STS \cite{CameronWebb} to
{\em path-graphs} of $q$-Steiner systems induced by quasigroups.
Section~\ref{pg} gives the rather complicated definition of a path graph. We
then lift properties of chains in infinite STS \cite{CameronWebb} to
$q$-Steiner systems. We give examples where all the path graphs over
algebraically closed sets are infinite (Lemma~\ref{disjointness}) and the
systems are decomposed as unions of `fans', Lemma~\ref{fancover} which
generalizes the decomposition by chains in the triple system case.

Rather than {\em ad hoc} examples, we provide a method to construct first
order theories and thus infinite families of countable models exhibiting
designated combinatorial properties. The countable models of these strongly
minimal theories are arranged in a {\em tower}, a countable increasing
sequence $\langle M_i: i<\omega\rangle$ with $M_n \prec M_{n+1}$. The
structure of $G_{M_0}(a,b)$ depends heavily on whether $\acl_{M_0}(\emptyset)
=\emptyset$. In various cases $G_{M_0}(a,b)$ may have only finite cycles,
only infinite cycles or a mixture. Lemma~\ref{apmu2} constructs a class of
$q$-Steiner
systems where all paths  are infinite. 

In Section~\ref{unif}, we generalize the construction by Cameron and Webb of
uniform cycle graphs in STS in \cite{CameronWebb} to $q$-Steiner graphs.  For
this we need to find $2$-transitive Steiner systems. Although the only finite
2-transitive STS are $PG(d, 2)$ and $AG(d, 3)$ \cite{KeyShult},  in every
infinite cardinality we point to $2$-transitive and so uniform $3$-Steiner
systems (Fact~\ref{Huniform}) and construct $2$-transitive $p^n$-Steiner
systems for every prime power (Theorem~\ref{restrictk0}).
 For this we must alter  the parameters
for a Hrushovski construction that are described in Notation~\ref{hruclass}.
We ask questions that depend on our construction at various places in the
text. But we conclude in Section~\ref{questions} by raising several questions
which should require more combinatorial methods.



In the remainder of the introduction we give further context and background.
The Barbina-Casanovoas examples  \cite{BarbinaCasa} are extremely
complex\footnote{While it is extremely likely that the theories of
\cite{HorsleyWebb} have similar complexity, that has not been worked out.}
from the viewpoint of the stability hierarchy ($TP_2$ and $NSOP_1$), while
strongly minimal are the simplest; algebraic closure imposes a matroid
structure on each model. A first order theory $T$ is strongly minimal if
every definable subset of every model of $T$ is finite or co-finite. Three
prototypical examples are the theories of: the integers with successor,
rational addition, and the complex field.
 Zilber conjectured these examples were canonical; each
such geometry was discrete, vector space like, or field like.  Hrushovski
refuted this conjecture by an intricate extension of Fra{\" \i}ss\'{e}'s
construction of countable homogeneous universal models. The resulting
`generic' model is `less' homogeneous that those built by  Fra{\" \i}ss\'{e}.
By use of a function $\delta$, a class of finite {\em strong} substructures
is obtained and only isomorphisms among them are required to extend.


 A
{\em linear space} is collection of points and lines such that two points
determine a line, a minimal condition to call a structure a geometry. A
linear space is a Steiner $k$-system if every line (block) has cardinality
$k$. We showed in Section 2 of \cite{BaldwinPao} that linear spaces can be
naturally formulated in a one-sorted logic with single ternary `collinearity'
predicate and proved the existence of strongly minimal Steiner $q$-systems
for every prime power. These theories are model complete and satisfy the
usual properties of
 counterexamples to Zilber's trichotomy conjecture: Their $\acl$-geometries are
 non-trivial, not locally modular, and the theory cannot interpret an infinite group.

%
Much of the history of Steiner systems interacts with the general study of
non-associative algebraic systems such as quasigroups. A quasigroup is a
structure with a single binary operation whose multiplication table is a
Latin Square (each row or column is a permutation of the universe)
\cite{Steinpnas}.
 Drawing on universal algebra and combinatorics,
 we \cite{BaldwinsmssII} found that the restriction  to prime power
 cardinality of the universe
 that is essential
 for the
existence of quasigroups coordinatizing finite $q$-Steiner systems is
replaced by prime power block length  for (necessarily infinite) strongly
minimal Steiner systems.

The III in the title indicates  the heavy reliance for details on
\cite{BaldwinPao,BaldwinVer,BaldwinsmssII}.
 The novelty here is that applying
 these methods to combinatorial issues requires new
  changes in the parameters of the construction.
We acknowledge helpful discussions with Joel Berman, Omer Mermelstein,
Gianluca Paolini, and Viktor Verbovskiy.

    \section{Background}\label{bg}
     \numberwithin{theorem}{subsection}

Hrushovski's `flat geometries' \cite[Definition 6.2]{BaldwinPao}
have generally been regarded from the standpoint of their creation:  as an
undifferentiated class of pathological structures designed as
counterexamples. However, there are ternary fields, Steiner systems and
quasigroups are among them. \cite{Baldwinasmpp, BaldwinPao, BaldwinsmssII,
BaldwinVer} shows that structural distinctions arise by fixing a class $\bU$
 of permissible choices for  the function $\mu$.  In fact, the family of `Hrushovski constructions of
strongly minimal sets' depend  on five parameters. We list them here for
reference; we demonstrate below that modifying  these parameters can produce
strikingly different behavior.
%


\begin{notation}\label{hruclass}{\rm
A quintuple $(\sigma,\bL^*_0, \bL_0,\epsilon,{\bf U})$ determines a {\em
Hrushovski sm-class}. $\bL^*_0$ is a collection of finite structures in a
vocabulary $\sigma$, not necessarily closed under
substructure\footnote{$\bL^*$ contains $\sigma$ structures of arbitrarily
cardinality.
$\bL^*_0$  was closed under substructure in \cite{Hrustrongmin} but not
here.}. $\epsilon$ is a function from a specified collection of finite
$\sigma$-structures to natural numbers satisfying the conditions imposed on
$\delta$ in Definition~\ref{defdelrank}. $\bL_0$ is a subset of $\bL^*_0$
defined using $\epsilon$. From such an $\epsilon$, one defines notions of
$\leq$, primitive extension, and good pair. Hrushovski gave one technical
condition on the function $\mu$ counting the number of realizations of a good
pair that ensured the theory is strongly minimal rather than $\omega$-stable
of rank $\omega$.  Fixing a class ${\bf U}$ as the collection of functions
$\mu$ satisfying a specific condition provides a way to index a rich group of
distinct constructions. As  explained in Definition~\ref{yax}, from
$\bL_0,\epsilon$ and $\mu \in {\bf U}$, one defines an amalgamation class
$(\bL_\mu,\leq)$  of finite structures and an associated class of infinite
structures $\hat L_\mu$
(For any collection $\bL$ of finite structures, we write
$\hat \bL$  for the
collection of direct limits of structures in $\bL$.).   
Thus, one obtains a strongly minimal theory  $T_\mu$ of a generic structure
$\Gscr_\mu$, that describes the `existentially closed' members of
 ${\hat \bL_\mu}$.}
\end{notation}

We show how modifications of the most basic Hrushovski construction provides
examples of Steiner systems.
 \cite[2.1, 2.2]{BaldwinPao} summarises the role of strongly
minimal sets in model theory and the bi-interpretability of a one-sorted
(used here) and two-sorted approach to Steiner system. \cite{Baldwinfg}
provides a somewhat outdated survey of vastly wider study of modifications of
the construction to study e.g. fusions, `bad' fields, Spencer-Shelah random
graphs and higher levels of stability classification.
 In
Section~\ref{homconfus} we sketch the relation between the combinatorial and
model theoretic literature. Section~\ref{hrucon} outlines the general setting
of the so-called {\em ab initio} Hrushovski construction, generated from a
collection of finite structure, emphasizing the parameters that can be varied
to get specific behaviors. Remark~\ref{hrusp} reminds us of the original
context; Section~\ref{linsp} lays out the notation for studying linear
spaces.

For convenience, one usually specifies in $\bL^*$ that the relations are
symmetric; but to reach important cases such as linear spaces, quasigroups,
and Steiner systems one adds the relevant axioms to this starting point. We
axiomatize $\bL^*$ with $\forall\exists$ sentences to create quasigroups.
Working in linear spaces with a `geometric' $\epsilon$ in \cite{Paoliniwstb}
is vital to obtain Steiner systems. In this paper, to obtain Steiner systems
which are (e.g. anti-Pasch, $\infty$-sparse, $2$-transitive) we both vary the
class $U$ of admissible $\mu$-functions and change the way that the class of
finite structures $\bL_0$ is determined by the relevant $\delta$ playing the
role of $\epsilon$.

\subsection{Constructing `Generic' models}\label{homconfus}

The constructions in \cite{BarbinaCasa,HorsleyWebb} and in this paper are
related generalizations of Fra{\" \i}ss\'{e}'s construction of a generic
model $M$ from a collection $\bL_0$   of finite structures in a {\em finite
relational vocabulary}  that is closed under substructure. A {\em generic
model} for a class $\bJ$ of finite structures is one that is homogeneous and
embeds all members of $\bJ$. Both `homogeneous' and `embed' change with the
author. In the Fra{\" \i}ss\'{e} setting $\bJ$ is the set of substructures
(as in the next paragraph) of $M$.

For ease in following references, I use the following model theoretic
terminology.
 A
{\bf substructure} $B$ of a structure $A$ in a vocabulary $\sigma$ (list of
function and relation symbols)  is a $B\subseteq A$ that is i) closed under
the function symbols in $\sigma$ and ii) each $n$-ary relation $R^B$ is $R^A
\cap B^n$.


 A structure $A$ is (finitely) {\em ultra-homogeneous\footnote{This model
  theoretic usage dates from \cite{Woodrow4ultra}.}} if
every isomorphism between finitely generated substructures of $A$ extends to
an automorphism of $A$.   This term corresponds to `homogenous' in \cite[p.
2]{HorsleyWebb} (as noted there). I use homogenous in the usual model
theoretic sense:  a structure $A$ is (finitely) {\em homogeneous} if any two
sequences $\abar$ and $\bbar$ of length $n<\omega$ that satisfy the same
$n$-ary first order formulas are automorphic in $A$. This notion appears in
an essential way in the proof of Lemma~\ref{gettrans}.

 Fra{\" \i}ss{\'e}
constructed  ultrahomogeneous, quantifier eliminable, and
$\aleph_0$-categorical structures in {\em finite relational vocabularies}.
His crucial hypotheses were joint embedding, amalgamation, and closure under
substructure.  As the construction was generalized to other notions of
substructure and possibly infinite vocabularies or with function symbols, two
hypotheses that were hidden by the finite relational vocabulary hypothesis
became evident: uniform local finiteness \footnote{In model theoretic terms,
a structure is locally finite if every finitely generated substructure is
finite (uniformly if there is a function $f$ such that an $n$-generated
structure has less than $f(n)$ elements). In \cite{HorsleyWebb} a structure
$M$ is called finitely generated with respect to a class $\bK$ of finite
structures if every finite subset of $M$ is contained in a member of $\bK$.
Our generic is locally finite in that sense but not in the model theoretic
sense.  While the generic in \cite[\S 4]{BarbinaCasa} is locally finite in
both senses.
Consider the infinite $3$-generated chains
in Section~\ref{pg} and the proof that $T^*_{\rm sq}$ is not small in
\cite[Theorem 3.3]{BarbinaCasa}.} (for $\aleph_0$-categoricity) and only
countably many finite structures (for the generic to be countable).

The three amalgamation constructions discussed here can best be compared in a
more abstract framework. Consider a countable collection $(\bL^*_0,\leq)$ of
finite structures where $\leq$ is a partial order refining substructure and
$\bL^*_0$ is defined by a collection of first order (usually universal)
sentences in the vocabulary of $\bL^*_0$. $M$ is $\leq$-homogeneous if $A,B
\leq $ implies they are automorphic. And $M$ is $\bL_0$-universal if there is
an isomorphism $F:A\rightarrow M$ with $f(A) \leq M$.  Sufficient conditions
are given so there is a structure $M$ which is $\leq$-homogeneous and
universal for $\bL^*_0$. There is no requirement that the language is
relational.  Each of the three constructions discussed here interpret $\bJ$
and $\leq$ in a different way.

\cite[\S 4]{BarbinaCasa} takes $\bL^*_0$ as the class of finite STS and
$\leq$ as substructure.
They construct a `generic' (a Fra{\" \i}ss{\'e} limit), which is a prime
model of their separately constructed $T^*_{\rm sq}$, the model completion of
the theory of all Steiner quasigroups.

\cite[Theorem 1]{HorsleyWebb}  generalizes this situation by taking $\bL^*_0$
as the class of `good $\bF$-free structures' (omit a collection $\bF$ of
finite nontrivial STS) finite triple systems.
The key distinction from our work is that those authors restrict their
amalgamation class to Steiner triple systems and use  or prove the
combinatorial fact that finite partial Steiner triple systems extend to
finite Steiner triple systems in the fixed $\bL^*_0$.

 In contrast,
we prove amalgamation by applying a general procedure due to Hrushovski to an
ambient class of finite linear spaces, bound line length by the
$\mu$-function (so partial Steiner systems), and obtain uniform line length
by  the `everything that can happen does'
mantra of amalgamation constructions. It is routine \cite[Section
2.1]{BaldwinPao} that strongly minimal linear spaces have bounded line-length
and cofinitely many lines have the same length; the existential completeness
of the generic model implies all lines have this maximal length
%

  We introduce the class $\bL_0
\subseteq \bL^*_0$ which is defined by properties of a pre-dimension function
$\delta$ and then further restrict with an `algebrizing function' to
$\bL_\mu$. Both $\delta$ and $\mu$ limit membership in $\bL_\mu$ to obtain
strongly minimal Steiner $q$-systems and in Section~\ref{consquasi}
$\bK^q_{\mu',V}$ for strongly minimal quasigroups that induce such systems.
Crucially, in that section we drop the requirement that $\bL^*_0$ is closed
under substructure.

For \cite{BarbinaCasa}, the generic model is prime but there is no countable
saturated model and the theory is not stable.  Our generic is saturated (so
model theoretically homogeneous) and the theory is strongly minimal (in
particular, $\omega$-stable).  But our generics are not ultrahomogeneous but
only $\leq$-homogeneous. In \cite{BarbinaCasa}, $T^*_{Sq}$ is the model
completion of theory of all Steiner quasigroups. Our
  theories $T_\mu$ are model complete. But they are not the model
completion of the $\bL_0$ in \cite{BarbinaCasa}; our $\bK^q_V$
(Section~\ref{consquasi}) is a much restricted class of finite quasigroups.
The generics of \cite{HorsleyWebb,BarbinaCasa} are locally finite; ours are
not.

%
Thus, from a model theoretic standpoint the strong minimality\footnote{The
easier $\omega$-stable step of the Hrushovski construction does not yield
Steiner systems.} distinguishes our example; while {\em from the
combinatorial standpoint the extension from $3$ to $p^n$-Steiner systems is
central} is the main novelty.

\subsection{The Hrushovski framework}\label{hrucon}

The basic ideas of the Hrushovski construction are i) to modify the Fra{\"
\i}ss\'{e} construction by replacing substructure by a notion of strong
substructure, defined using a predimension $\delta$
(Definition~\ref{defdelrank}) so that independence with respect to the
dimension induced by $\delta$ is a combinatorial geometry\footnote{The
requirement that the range of this function is well-ordered is essential to
get the exchange property in the geometry; using rational or real
coefficients yields a stable theory and the dependence relation of forking
\cite{BaldwinShiJapan}.} and ii) to employ an {\em algebrizing function}
$\mu$ to bound the number $0$-primitive extensions of each finite structure
so that closure in this geometry is algebraic closure\footnote{In model
theory $a \in \acl_M(B)$ if there is a formula  and $\bbar \in B$ with
$\phi(a,\bbar)$ and $\phi(x,\bbar)$ has only finitely many solutions in
$M$.}.

 A Steiner $(t,k,v)$-system is a pair $(P,B)$ such that $|P| = v$, $B$ is a
     collection of $k$ element subsets of $P$ and every $t$ element subset
     of $P$ is contained in exactly one block.  Since we are primarily
     interested in infinite structures, we omit the $v$ unless it is
     crucial and so, by Steiner $k$-system  I mean Steiner $(2,k)$ system of arbitrary cardinality.
A {\em groupoid}
    (also called a {\em magma})  is a structure
$(A,*)$ with one binary function $*$.


Unfortunately, while the extensive literature on Hrushovski constructions
contains the same fundamental notions related in a fairly standard way, the
notation is not standard. So we quickly list our terminology.

We give an abstract formulation of the construction of generic model due to
\cite{KuekerLas}.  This provides a common framework for the Fra{\" \i}ss\'{e}
and Hrushovski constructions which does {\em not} require the class $\bL$ to
be closed under substructure and is essential in Section~\ref{consquasi}. For
the general discussion in this section we work in a finite relational
vocabulary $\sigma$.

\begin{notation}

\label{not0}

\begin{enumerate}
\item For any class $\bL$ of finite structures, $\hat \bL$ denotes the
    collection of structures of arbitrary cardinality that are direct
    limits\footnote{If $\bL_0$ is closed under substructure so is $\hat
    \bL_0$ and  $\hat \bL_0$ is axiomatized by a universal sentence in
    $L_{\omega_1,\omega}$ ($L_{\omega,\omega}$ if the vocabulary is
    relational.).} of models in $\bL$.
\item Let $\sigma$ be a finite relational vocabulary. A class $(\bL_0,
    \leq)$ of finite structures, with a transitive relation $\leq$ on
    $\bL_0 \times \bL_0$ is called {\em smooth} if $B\leq C$  implies
    $B\subseteq C$ and for all $B\in \bL$ there is a collection
    $p^B(\xbar)$ of universal formulas with $|\xbar| = |B|$ and for any
    $C\in \bL_0$ with $B\subseteq C$,

$$B\leq C \leftrightarrow C \models \phi(\bbar)$$
for every $\phi \in P^B$ and $\bbar$ enumerates $B$.

We write $B$ is strongly embedded in $C$ if an isomorphic image $B'$ of $C$
satisfies $B'\leq C$.

\item A structure $A$ is a $(\bL_0, \leq)$-union if $A = \bigcup_{n<
    \omega} C_n$ where each $ C_n\in \bL_0$    and $C_n \leq C_{n+1}$ for
    all $n< \omega$. If $A$ is a $(\bL_0, \leq)$-union, $B \subseteq A$, $B
    \in \bL_0$, we say $B\leq A$ if $B \leq C_n$ for all sufficiently large
    $n$.
    \item A structure $A$ is an $(\bL_0, \leq)$-{\em generic} or
        $\leq$-{\em homogeneous} if $A$ is a $(\bL_0, \leq)$-union and for
        any $B\leq C$ each in $\bL_0$ and $B\leq A$ there is a
        $\leq$-embedding of $C$ into $A$.
        \end{enumerate}
\end{notation}

$\leq$ is read `strongly embedded'. The crucial fact is:

\begin{fact}[\cite{KuekerLas}]\label{KL} If $(\bL_0,\leq)$ is a smooth class of countably
many finite structures  that satisfies $\leq$-amalgamation and $\leq$-joint
embedding there is a unique countable generic $\Gscr_{\bL_0}$ for
$(\bL_0,\leq)$
\end{fact}



%


\begin{axiom} \label{yax}
Let $\delta$ be a map from a collection of finite $\sigma$-structures
 into $\omega$.
Let ${\bL^*_0}$ be a collection of such structures closed  
under isomorphism.  We write $A\leq B$ if for every $C$ with $A \subseteq C
\subseteq B$, $\delta(C/A) \geq 0$.
  We require that $\bL^*, \hat \bL^*, \delta$
satisfy the following requirements.
First, ${\bf L_{0}} \subsetneq \bf L^*_{0} $ is the collection of finite $B$
such that:

\begin{enumerate}
\item $\delta(\emptyset) = 0$
\item If $B \in \bL_0$ and  $A \subseteq B$ then $\delta(A) \geq 0$.
\item If $A$, $B$, and $C$ are disjoint then $\delta(C/A) \geq
    \delta(C/AB)$.

\item $(\bL_0),\delta)$ admits canonical amalgamations in the following
    sense.

\end{enumerate}

\end{axiom}

\begin{definition}{\bf Canonical Amalgamation}\label{canamdef} For any class $(\bL_0,\epsilon)$,
    if $A\cap B = C$, $C\leq A$ and $A,B, C
    \in \bL^*$,  $G$ is a {\em free (or canonical) amalgamation},  $G = B \oplus_C  A$
    if
      $ G \in
    \bL^*$, $\epsilon(A/BC) = \epsilon(A/C)$ and $\epsilon(B/AC) = \epsilon(B/C)$
     Moreover, $\epsilon( A \oplus_{C} B)= \epsilon(A) +\epsilon(B) -
    \epsilon(C)$ and any $D$ with $C \subseteq D \subseteq A \oplus_{C} B$ is
    also free.  Thus, $B \leq G$.
    \end{definition}

Disjoint union is the canonical amalgamation for the basic Hrushovski
construction and Definition~\ref{defcanam}
  gives the appropriate
notion satisfying Axiom~\ref{yax}.5 for linear spaces. Axiom~\ref{yax}.2 can
be rephrased as: $B\subseteq C$ and $A \cap C = \emptyset$ implies
$\epsilon(A/B) \geq \epsilon(A/C)$; so we can make the following definition.

\begin{definition}\label{defd} Extend $\epsilon$ to $d: {\bbL}_0 \times
{\bL_0}
 \rightarrow \omega$ by for each $N \in {\bbL}_0$ and $A \subset_{<\omega} N$,
$d(N,A) = \inf \{\epsilon(B): A \subseteq B \subseteq_{\omega} N\}$, $d_N(A/B)
= d_M(A \cup B) -d_M(B)$. We usually write $d(N,A)$ as $d_N(A)$ and omit the
subscript $N$ when clear.
\end{definition}


%
%
%
%
%
%
%

What Hrushovski called {\em self-sufficient} closure is in the background.


\begin{definition}\label{strongdef}

	\begin{enumerate}[(1)]
\item For $N \in {\bbL}_0$ and $A \in {\bL}_0$, we say $A\subseteq N$ is
    strong in $N$ and write $A\le N$  if $d(N/A) \geq 0$.

  \item For any $A \subseteq B \in {\bL}^*$, the intrinisic
      (self-sufficient) closure of $A$, denoted $\icl_B(A)$ is the smallest
      superset of $A$ that is strong in $B$.
\end{enumerate}
\end{definition}
Note that in the current situation $\icl(B)$ is finite if $B$ is.
The following definition describes the pairs $B \subseteq C$ such that
eventually $\tp(C/B)$ will be an algebraic set (realized only finitely often).

\begin{definition}\label{prealgebraic} Let $A, B \in {\bL}_0$ with
 $A \cap B = \emptyset$ and $A \neq \emptyset$.
	\begin{enumerate}[(1)]
\item $B$ is a {\em  primitive extension} of $A$ if $A \leq B$ and there is
    no $A \subsetneq B_0 \subsetneq B$ such that $A \leq B_0 \leq B$. $B$
is a {\em  $k$-primitive extension} if, in addition, $\epsilon(B/A) =k$.

 We stress that in this definition, while $B$ may be empty, $A$ cannot be.
	\item  We say that the $0$-primitive pair $A/B$ is {\em good} if there
is no $B' \subsetneq B$ such that  $(A/B')$ is $0$-primitive. (This notion
was originally called a minimal simply algebraic or m.s.a. extension.)
	\item If $A$ is $0$-primitive over $B$ and $B' \subseteq B$ is such that we have that $A/B'$ is good, then we say that $B'$ is a {\em base} for $A$ (or sometimes for $AB$).  
	\item If the pair $A/B$ is good, then we also write $(B,A)$ is a {\em good pair}.
\end{enumerate}
\end{definition}

	\begin{definition}     \label{Kmu}	
	\begin{enumerate}[(1)]
	\item  Let $\Uscr$ be the collection of functions $\mu$ assigning to
every isomorphism type $\boldsymbol{\beta}$ of a good pair $C/B$ in $\bL_0$
a natural number
$\mu(\boldsymbol{\beta}) = \mu(B,C) \geq \epsilon(B)$. 

\item For any good pair $(B,C)$ with $B \subseteq M$ and $M \in \hat
    \bL_0$, $\chi_M(B,C)$ denotes the number of disjoint copies of $C$ over
 $B$ in $M$. A priori,  $\chi_M(B,C)$ may be $0$.

	\item\label{Kmuitem} Let $\bL_\mu$ be the class of structures $M$ in
$\bL_0$ such that if $(B,C)$ is a good pair  $\chi_M(B,C)  \leq
\mu((B,C))$.
\end{enumerate}	
\end{definition}

Up to this point, we have denoted the rank function by $\epsilon$ to indicate
it is being treated entirely axiomatically. {\em We switch to} $\delta$ to
emphasize that (Hrushovski's definition (Definition~\ref{hrusp}) or Paolini's
(Definition~\ref{defdelrank}) may be used but trust to context for the reader
to know which.

\begin{remark} [The basic  Hrushovski construction]     \label{hrusp}
In the original context \cite{Hrustrongmin}, $\sigma$ contains a single
ternary relation $R$ and  $\delta(A) =|A| - r(A)$ where $r(A)$ is the number
of triples $\abar$ from $A$ satisfying $R(\abar)$. $\bL^*$ is all finite
$\sigma$-structures and $\bL_0$ is  those $A\in \bL^*$ with  $\emptyset \leq
A$ and $\Uscr$ is as in Definition~\ref{Kmu} with that $\delta$.
\end{remark}

We have recalled the Hrushovski notion for context but this paper is entirely
about linear spaces.

\subsection{Linear Spaces}\label{linsp}
In this section we outline the adaptation of Remark~\ref{hrusp} that
generates most of the examples in this paper. For the remainder of the paper
we will deal at various times with two vocabularies $\tau$, with a single
ternary relation symbol, $R$, and $\tau' =\{H,R\}$  with a ternary relation
$H$,
which will be the graph of a binary function $*$. 

\begin{definition}\label{defls}
A $\tau$-structure $(M,R)$ is

\begin{enumerate}
\item a {\em $3$-hypergraph} if $R$ holds only of distinct triples and in
    any order.
\item a {\em linear space} if it is a $3$-hypergraph in which two points
    determine a unique line. That is, each pair of distinct points in
    contained in unique maximal $R$-clique (line). That is, all triples
    from the line satisfy $R$.
    \item A  linear space is a {\em $k$-Steiner system} if all lines have
        the same length $k$.
        \end{enumerate}
\end{definition}

Thus our finite structures will in general be partial $k$-Steiner systems
(lines may not have full length) for some $k$. We use the words `block' and
`line' interchangeably and often fail to distinguish when the line has full
length. When this is important, we may write clique to denote a subset of a
line, i.e.,  a maximal clique.

 \begin{definition}\label{lines}
 \begin{enumerate}
 \item  For $\ell \subseteq A$, we denote the cardinality of a clique
     $\ell$ by $|\ell|$, and, for $B \subseteq A$, we denote by $|\ell|_B$
     the cardinality of $\ell \cap B$.
   \item We say that a non-trivial line $\ell$ contained in $A$ is {\em
       based in} $B \subseteq A$ if $|\ell\cap B| \geq 2$, in this case we
       write $\ell \in L(B)$.
\item	\label{nullity} The {\em nullity of a line} $\ell$ contained in a
    structure $A \in \mathbf{K}^*$ is:
	$$\mathbf{n}_A(\ell) = |\ell| - 2.$$
\end{enumerate}
\end{definition}

Now we define our geometrically based pre-dimension function
\cite{Paoliniwstb}.

\begin{definition}\label{defdelrank} We define the appropriate $\bK^*$ and $\bK_0$.
\begin{enumerate}
 \item Every $(A,R) \in  \bK^*$ is a finite linear spaces.

 \item For $(A,R) \in  \bK^*$ let:
	$$\delta(A) = |A| - \sum_{\ell \in L(A)} \mathbf{n}_A(\ell).$$
\item  Moreover $(A,R) \in  \bK_0$ if
for any $ A' \subseteq A, \delta(A') \geq 0\}$. \item $(\bK_0,\delta)$
satisfies the conditions on $\epsilon$  given in  Section~\ref{hrucon}
\end{enumerate}
\end{definition}

The explicit definition of the free amalgamation {\em in this context} is:

\begin{definition}\cite[Lemma 3.14]{BaldwinPao}\label{defcanam} Let $A \cap B =C$
 with $A,B,C \in \bK_0$.
We define $D := A \oplus_{C} B$ as follows:

\begin{enumerate}[(1)]
	\item the domain of $D$ is $A \cup B$;
	\item
a pair of points  $a \in A - C$  and $b \in B - C$ are on a non-trivial
line $\ell'$ in $D$ if and only if there is line $\ell$ based in $C$ such
that $a \in \ell$ (in $A$) and $b \in \ell$ (in $B$). Thus $\ell'=\ell$ (in
$D$).
\end{enumerate}
\end{definition}

%
%
%
%
%

We single out a type of good pair that provides the line-length
invariant for the Steiner systems.

\begin{notation}[Line length]\label{linelength} We write $\boldsymbol{\alpha}$ for the isomorphism
type of the good pair $(\{b_1,b_2\},a)$ with
 $R(b_1,b_2,a)$. 
Lemma 5.18
of \cite{BaldwinPao}) implies lines in models of $T_\mu$ have length $k$ if
and only if $\mu(\boldsymbol{\alpha}) = k-2$.
 \end{notation}

If one restricts the counting functions to $\Uscr$ (Definition~\ref{Kmu}),
Steiner triple systems are excluded. Since they are a key topic, the
$\Uscr$ is slightly altered from Definition~\ref{Kmu} to admit them.

\begin{definition}[$\Uscr^{\rm ls}$] \label{Kmu'}  Let $\Uscr^{\rm ls}$ be the
 collection of functions $\mu$ assigning to every isomorphism type $\boldsymbol{\beta}$
of a good pair $C/B$ in $\bK_{0}$ a number
$\mu(\boldsymbol{\beta}) = \mu(B,C) \geq \delta(B)$. 
	\begin{enumerate}[(i)]
	\item a number $\mu(\boldsymbol{\beta}) = \mu(B,C) \geq \delta(B)$, if $|C-B|\geq 2$;
	\item a number $\mu(\boldsymbol{\beta}) \geq 1$ (rather than $2$), if $\boldsymbol{\beta}) = \boldsymbol{\alpha}$.
\end{enumerate}
\end{definition}




%


\section{Omitting configurations in Steiner triple systems}\label{omcon}
 \numberwithin{theorem}{subsection}

There is a long history of studying finite Steiner triple systems that omit
specific configurations, e.g. Pasch.  The concept is formalized as follows
(\cite{Fujiwaramitre}).

\begin{definition} Let $X$ be finite partial Steiner system.  A Steiner
system $(M,R$) is {\em anti-X} if there no embedding of $X$ into $M$.
\end{definition}

 The notion of an $\infty$-sparse
system uniformizes these anti-$x$ constructions \cite{Fujiwara,Chicoetal}. We
derive such results for infinite Steiner triple systems by variants on our
general construction. We first find specific amalgamation constructions that
give the strongly minimal Steiner systems omitting target configurations  by
varying the class $U$ of acceptable bounds on algebraicity. We next obtain
$\infty$-sparseness in Section~\ref{spconfig}, by enforcing the uniformity
with $\delta$ and with more drastic restrictions on the class $\bK_{0}$.

\subsection{Anti-Pasch and Anti Mitre}\label{antipa}

 \begin{figure}[ht]\label{Paschdia}
 \begin{center}
\includegraphics[height=2in]{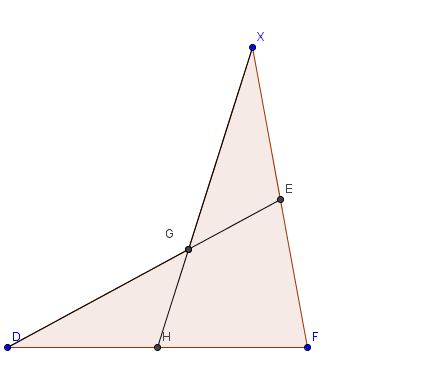}
\caption{Pasch configuration}
\end{center}
\end{figure}

We begin by examining the connection between the Pasch configuration
\cite{Fujiwara} and the group configuration from model theory.
Diagram~\ref{Paschdia} is known in the study of Steiner triple systems as the
{\em Pasch configuration}. This same diagram, interpreting the lines as
representing algebraic closure, is known to model theorists as  the {\em
group configuration}: in that context the $\acl$-dimension of the set of 6
points is 3; any triple of non-collinear points are independent; each point
has $\acl$-dimension 1, and each line has $\acl$-dimension 2. Hrushovski's
proof, described for the Steiner system case in \cite[Corollary
6.3]{BaldwinPao}, that no $T_\mu$ interprets an infinite group originated the
model theoretic argument that the group configuration in the algebraic
closure geometry implies the existence of a definable infinite group.
We give a more direct argument for:

\begin{fact}\label{noinfsubgr} The strongly minimal quaigroups
whose existence is proven in Section~\ref{consquasi}, have no infinite
definable {\em associative} subquasigroup.
\end{fact}

Proof. Let $G$ be a definable  infinite subquasigroup of $\Gscr_\mu$ with
associative multiplication that is generated by three algebraically
independent elements, say $D,G,H$ as in Figure~\ref{Paschdia}. Now $DG =E$
so, by associativity, $F(DG) = FE =X$. Similarly, $H=FD$ implies $(FD)G = HG
=X$ so the lines $HG$ and $FE$ intersect in $X$. In any $T_\mu$ the algebraic
closure dimension of a closed subset $A$ is $d(A) =\delta(A)$. So if $A$ is
the six points of the configuration we should have $\delta(A) \geq d(A) =3$.
But the actual calculation\footnote{Section 4.2 of \cite{Hrustrongmin}
describes combinatorial geometries that calculate the dimension of a union of
closed subsets by exclusion-inclusion principle as `flat'.} gives $\delta(A)
=2$. So the Pasch configuration
 is omitted. $\qed_{\ref{noinfsubgr}}$



%

In particular, a strongly minimal quasigroup constructed in this way can
never be a group. The associative law forces the intrinsic closure $\icl(P)$
of three algebraically independent elements (which should have $d(\icl(P)
\geq 3$) to have dimension $2$. Nevertheless, in general there will be many
realizations of a Pasch configuration $P$ in a strongly minimal Steiner
triple system constructed as in \cite{BaldwinPao}, since $\delta(P) = 2 \geq
0$. Indeed any pair of points extends to a Pasch configuration in the generic
model. Fact~\ref{noinfsubgr} shows in general that configuration cannot
extend to an infinite subquasigroup.

\cite{HorsleyWebb} suggested that anti-Paschian STS might be an amalgamation
class.  We have't shown that;  we construct smaller amalgamation classes of
anti-Paschian STS. We need the following notion.


\begin{definition}[$R$-closure]\label{Rcl} Let $(M,R)$ be a
$\tau$-structure. We define the $R$-closure, $\cl_R(X)$, for $X \subset M$.
Define inductively $X= X_0$ and for each $n$,  $c \in X_{n+1}$ if
 $a,b \in X_n$ and $R(a,b,c)$. Now $\cl_R(X) = X_N$,  where $N$ (possibly
$\omega$) is where the inductively defined sequence $X_n$ terminates.  A
set $X$ is $R$ independent if no element is in the $R$-closure of the
others.
\end{definition}

 \begin{lemma}\label{omitPascham} The subclass of $\bK^P_0$  of those finite
 structures with $3$-element lines that  omit the Pasch configuration
 satisfies amalgamation.
 \end{lemma}

 \begin{proof}
We can reformulate the problem by setting $\boldsymbol{\rho}$ as the
isomorphism type of the good pair $(A/B)$ in Figure~\ref{Paschdia}, taking
$\{F,H\}$ as the base $B$  and $A =\{X,D,E,G\}$ as a good extension. We use
the standard $\bK_0$ for linear space. But we modify a $\mu \in \Uscr$ by
setting $\mu(A/B) =0$. We must show $\bK_\mu$ has amalgamation.

Fixing notation as in the proof of amalgamation in \cite[5.11]{BaldwinPao},
consider structures with $(E/D)$ a good pair, $D\subseteq F$ and all in
$\bK_0$; we want to amalgamate $F$ and $E$ over $D$. Note that every
 non-trivial line that intersects $E-D$ is contained in $E$ and has two
elements in $E-D$. This holds, as if the line intersects $F-D$ then it has
$4$ points  by Definition~\ref{defcanam}. But, if it intersects $D$ in $2$
points $E$ is not primitive over $D$.  Thus $F$ is $R$-closed in $G$.
The key property of the Pasch configuration is that each point not in the
base is on a $3$-element line that intersects the base.  This implies that
if there is an embedding of the Pasch configuration $P$ in $G$, the image
of the base  $EH$ is contained in $D$. (Otherwise there would be a line
from $F-D$ to $E-D$.) But since the Pasch configuration is $R$-generated by
the base along with any other point, we have $A \subseteq F$ if $A \cap F
\neq \emptyset$ and $A \subseteq E$ if not. Either violates the hypothesis
that $F$ and $E$ omit the Pasch configuration.
\end{proof}

Applying Lemma~\ref{omitPascham},
\begin{corollary}\label{omitPasch} Fix a $\mu \in \Uscr$
with $\mu({\boldsymbol{\alpha}})=1$ and
  $\bK_\mu
\subseteq \bK^P_0$.
 Each model of $T_\mu$ is a strongly minimal anti-Pasch Steiner triple
 system. As usual, varying  $\mu$ yields $2^{\aleph_0}$ distinct
 families.
\end{corollary}

Similar arguments construct anti-mitre and anti-mia configurations.
The two  configurations are shown in Figure~\ref{mitreconf}. Letting $abc$ be
the bottom line, $c'b'a'$ the middle, and $x$ the vertex, the diagram
represents the left self-distributive law:
$$x(ab)= (xa)(xb).$$

Namely the self distributive law implies naming $a'$ as $xc$ and $c'$ as
$xa$ the lines $ac', bb', ca'$ intersect at $x$.  
 This
$(5,7)$-configuration \cite{Fujiwara} is called a {\em mitre\footnote{In the
diagram, $x$ is the top point. Label the middle line $a,b,c$ and the bottom
line $c',b',a'$.

Diagram taken from \cite{CaFaPa}.}};  The only other $(5,7)$-configuration,
({\em mia}), is obtained by  adding a point between the two points on the
base of the Pasch configuration and creating a new line.  By constructing
$\infty$-sparse configurations below we simultaneously omit the  Pasch,
mitre, and mia configurations.


 \begin{figure}[ht]\label{mitreconf}
 \begin{center}
\includegraphics[height=2in]{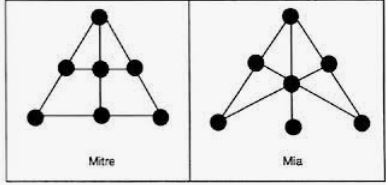}
\caption{Mitre and mia configurations}
\end{center}
\end{figure}

\begin{corollary}\label{omitmitre}
 There are anti-mitre and anti-mia Steiner triple systems in
every infinite cardinality. The examples are strongly minimal.
\end{corollary}

\begin{proof} Using the first paragraph of the proof of Lemma~\ref{omitPasch},
we show the subclass $\bK^M_0$ of $\bK_0$  consisting of  those finite structures
 with $3$-element lines that  omit the mitre (or those omitting the mia)
configuration satisfies amalgamation.
 The argument there that the base is contained in
 $D$, here yields only that two points of the base are in $E-D$. Say $a \in
 F$ and $b,c \in E-D$.  We violate $F$ closure unless the point $F-E$ is
 $c'$. Now if the pivot $x$ is in $E-D$, we violate the $R$-closure of $F$.
 But if $x\in F$ and $a'$ or $b'$ is in $F$, $a'xa$ or $b'xb$ violates that
 $E$ is primitive over $D$. While if either is in $E-D$, $a'b'c'$ violates
  $R$-closure of $F$.  The proof of the mia case offers nothing new.
  \end{proof}

  Thus we construct structures which have no instances of associativity or
  self-distributivity anywhere and every left multiplication
  by an element not on a line  fails to
 preserve lines.

%

\subsection{Sparse Configurations in 3-Steiner systems}\label{spconfig}

 In, for example \cite[page 116]{Chicoetal}, an $(n,n+2)$ configuration in
a Steiner triple system (STS) is a substructure $(A,R)$ of $n+2$ points
with $n$ lines.
That is, $\delta(A) =2$.  They  say a system is {\em
$\infty$-sparse} if there are no $(n,n+2)$ configurations with $n\geq 4$.
We reformulate `sparse' in terms of $\delta$.

\begin{definition}\label{sparsedef} A Steiner triple system $(M,R)$ is $\infty$-sparse if
there is no $A\subseteq M$ with $|A|\geq 6$ and $\delta(A) =2$.
\end{definition}

Note that   the Pasch, mitre, and mia configurations are all forbidden
 in an $\infty$-sparse STS.
\cite{Chicoetal} construct by a four page inductive construction of finite
approximations, $2^{\aleph_0}$ non-isomorphic countable $\infty$-sparse
systems.  
We modify the construction in \cite{BaldwinPao} by restricting $\bK_0$ to
$\bK^{sp}_0$ to get $\infty$-sparse STS of every infinite cardinality.

\begin{definition}\label{sparac}
Let $\bK^{sp}_0$ be the subclass of $\bK^*$ (linear spaces) such that for
every $B\subseteq A$: $$(\#)\  |B| > 1\ \rightarrow\ \delta(B) > 1\  \&
    \   |B|
    > 3 \rightarrow \delta(B) > 2.$$

Take  ${\bf U}$ as  $\Uscr^{sp}$, those
 $\mu\in \Uscr$ which can be achieved in $\bK^{sp}_0$.

\end{definition}

Condition $\#$ implies there are no $4$ element lines in a member of
$\bK^{sp}_0$ so  if $\mu \in \Uscr^{sp}$,  $\mu(\boldsymbol{\alpha}) =1$ and
the generic model will be a Steiner triple system.

%
%
%

%
%
%

\begin{theorem}\label{Sap} The system $(\bK^{sp}_0,\leq)$ has $\leq$-amalgamation. And so
for any $\mu \in \Uscr$, $\bK^{sp}_\mu$ has $\leq$-amalgamation.
\end{theorem}

Proof. Let $A, B, C\in \bK^{sp}_0$ with  $C \leq  A$ and  $C \leq  B$. Linear
space amalgamation (Definition~\ref{defcanam}) cannot introduce any relation
between $A-C$ and $B-C$, as this would produce a $4$-element line. But then
it is clear that $\#$ is preserved in the amalgam. We use the first clause of
$\#$ to avoid $B$ with $\delta(B) =1$.  Now the proof from \cite{BaldwinPao}
applies to give   amalgamation for $\bK_\mu$ if $\mu \in \Uscr^{sp}$.
$\qed_{\ref{Sap}}$

\begin{theorem}\label{getsparse}  There are continuum many $\mu$ such that
\begin{enumerate}
\item $T_\mu$ is strongly minimal (so $\aleph_1$-categorical);
\item Every model of $T_\mu$ is an $\infty$-sparse Steiner triple system;
\item $T_\mu$ has countably many countable models. \end{enumerate}

    \end{theorem}

 Proof. As in \cite{BaldwinPao}, for any $\mu$ satisfying
 Definition~\ref{sparac}, the associated $T_\mu$ is a Steiner triple
 system. But by omitting $A$ with $\delta(A) =2$ and $|A|\geq 6$, the
 structure is $\infty$-sparse. $\qed_{\ref{getsparse}}$

\section{Constructing strongly minimal quasigroups}\label{consquasi}
  \numberwithin{theorem}{section}

  While \cite{BaldwinPao} shows there are strongly minimal $k$-Steiner
systems for every $k$, \cite{GanWer,BaldwinsmssII} imply that there can be
quasigroups only when $k$ is a prime power. Our strongly minimal
$k$-Steiner systems $(M,R)$ can admit a definable `truly' binary function
\cite{BaldwinVer} only under very strong additional hypotheses on $\mu$
(\cite[Theorem 0.2]{BaldwinVer}). Nevertheless, there are strongly minimal
quasigroups which induce $k$-Steiner systems when $k$ is a prime power. For
this result we need the generality of Fact~\ref{KL}, as we will axiomatize
$\bL^*_0$ ($\bK^q_V$ here)
 with $\forall\exists$-sentences. We sketch a different proof than that
detailed  in \cite{BaldwinsmssII} of the existence of strongly minimal
quasigroups.

 The coordinatizing result rests primarily on work of
\cite{GanWer,Steinpnas,Swierfreealg} and others who achieved a
`coordinatization' of such Steiner systems by quasigroups. The contribution
here is  that  although, for $k>3$, because the Steiner system never
interprets a quasigroup \cite{BaldwinsmssII}, this coordinatization is not
a bi-interpretation,
 we can in fact demand for  $k= q =p^n$  the existence of a Steiner
 $k$-system that is interpreted in a strongly minimal quasigroup. The key
 to this
  is the relationship of so-called $(2,k)$ varieties \cite{Padchar,GanWer}
to a two-transitive finite structure and thus eventually to the
reconstruction of a finite field.  Following \cite{GanWer2} we call the
quasigroups which arise when $k$ is a prime power $q$, {\em block algebras}.

A variety is a collection of algebras (structures in a vocabulary with only
function/constant symbols and no relation symbols) that is defined by a
family of equations. The essential characteristic of the equational theories
below is that each defining equation involves only two variables.  In
particular, none of the varieties are associative.


\begin{definition}\label{defrelvar}\cite{Smithlec}
 A 
quasigroup  $(Q,*)$ is a groupoid\footnote{In the background literature on
quasigroups, a {\em groupoid} is simply a set with a binary operation. So,
I use this notation although it is no longer common.} $(A,*)$ such that for
$a,b\in  Q$, there exist unique elements $x,y \in Q$ such that both
$$a * x = b, y * a = b.$$
 \end{definition}

 The general notion is a universal Horn class, not a variety. But an $(r,k)$
 variety of groupoids is a quasigroup \cite{Quackquac}.

 \begin{definition} \cite{Padchar} \label{rkdef} \begin{enumerate}
 \item The variety $V$ is an $(r,k)$ variety if every $r$-generated
     subalgebra of any $A \in V$ is isomorphic to the free $V$-algebra on
     $r$ elements and has cardinality $k$.

  \item A {\em Mikado} variety  \cite[128]{GanWer} is $(2,q)$-variety with
      all fundamental operations binary and with an equational base of
      $2$-variable equations.
\end{enumerate}
 \end{definition}

 Thus, `Mikado' picks out those $(r,k)$-varieties that are really determined by
 their free algebras on $2$-generators.

 \begin{fact} \cite{GanWer} \label{blalg}
 Given a (near)-field\footnote{A near-field is an algebraic structure
     satisfying the axioms for a division ring, except that it has only one
     of the two distributive laws.} $(F,+,\cdot,-,0,1)$ of cardinality $q$
     and a primitive element $a \in F$, define a multiplication $*$ on $F$
     by $x*y = y+ (x-y)a$. An algebra $(A,*)$ satisfying the 2-variable
     identities of $(F,*)$ is in a $(2,q)$-variety of  {\em block algebras}
     over $(F,*)$.\end{fact}

This is one of 5 equivalent characterizations of an $(r,k)$  variety in
\cite{Padchar}. Obviously, the collection of $r$-generated subalgebras $A \in
V$ form a Steiner $(r,k)$-system; we need a third characterization: the
automorphism group of any $r$-generated algebra is strictly (i.e. sharply)
$r$-transitive.

Fix two vocabularies $\tau = \{R\}$ and  $\tau'$ with two ternary relations
symbols $R,H$. For each Mikado $(2,q)$-variety $V$ of quasigroups,   we
construct a strongly minimal theory of quasigroups (in $V$) that induce
$q$-Steiner systems. We use $H$ as the graph of the quasigroup operation in
$V$, $*$, to make our amalgamation class contain only finite structures (as
in \cite{BarbinaCasa}).  But $R$ is the ternary relation of collinearity.
Considering the general context of Notation~\ref{not0}, there are two
innovations in taking the $\bL_0$ as $\bK^q_V$: i) in each finite structure
in $\bK^q_V$ every line has $q$ points; ii) and $\bL_0 = \bK^q_{0,V}$
 for a $\tau'$-structure $A$, $\delta_{\tau'}(A') = \delta_{\tau}(A'\myrestriction \tau) $
so the amalgamation problem reduces to the known solution for
$\tau$-structures. We define the base class $\bK^q_{0,V}$ of finite
structures as follows.

\begin{definition}\label{K'mu}[$\bK^q$]
Fix a prime power $q$ and a Mikado variety $V$ of quasigroups (e.g. a block
algebra from Fact~\ref{blalg}) such that $F_2$, the free algebra in $V$ on
$2$ generators has $q$ elements. Let $\bK^q_V$ be the
collection\footnote{Clearly $V$ determines but there are distinct $V$ with
 $|F_2(V)| =q$.} of finite $(H,R)$-structures $A$ such that
\begin{enumerate}
\item $(A,R)$ is a linear space;

\item $(\forall a_1,a_2,a_3) H(a_1,a_2,a_3) \rightarrow R(a_1,a_2,a_3)$;

\item $(\forall a_1,a_2,a_3) [H(a_1,a_2,a_3) \wedge H(a_1,a_2,a'_3)]
    \rightarrow a_3 =a_3'$;
\item $(\forall a_1,a_2,a_3)(\exists b_1, \ldots b_{q-2}) [R(a_1,a_2,a_3)
    \rightarrow \bigwedge_{i\leq q-3}R(a_1,a_2,b_i)$;

\item $$(\forall x_1,x_2, \ldots x_{q})[ (R(x_1,x_2,x_i) \rightarrow
    \bigvee_{1 \leq i,j\leq q} x_i = x_j].$$
%
%
    \item If $A'\myrestriction R$ is a  maximal clique (line) with
        respect to $R$ (necessarily $|A'| =q$), $A'\myrestriction H$ is
        the graph of the free algebra $F_2\in V$.
    \end{enumerate}
    \end{definition}

    Note that Definition ~\ref{K'mu} implies that any {\em triple}
    satisfying $R$ in $A'\in \bK^q_V$ extends to a line in
    $A'$ of exactly length $q$. Since $V$ is axiomatized by $2$-variable
    equations,  if $A' \in \bK^q_V$, $A'\myrestriction H$ is
    the graph of an algebra in $V$. In the generic model {\em each pair} is
    included in a $q$-element line; but not in the finite structures.

\begin{definition}\label{defdelmu} Primitives, good extensions, and the permissible $\mu'$:
\begin{enumerate}
\item For a $\tau$-structure $(A,R)$ $\delta_{\tau} (A)$ is defined as for
    linear spaces in Definition~\ref{defdelrank}. Now for any $q$ and each
    $A' \in \bK^q_V$,
    let $A  = A'\myrestriction R $ and $\delta_{\tau'}(A' ) =
    \delta_\tau(A)$ and induce $\leq'$ from $\delta_{\tau'}$.

\item $\bK^q_{0,V} = \{A'\in \bK^q_V: \delta_{\tau'}(A')\geq 0\}$.
\item Define primitive extensions and good pairs in $\tau$ as usual using
    $\delta'_{\tau'}    $.

Let $\boldsymbol{\alpha}_q$ denote the isomorphism type of  $(\{c_1, c_2,
\ldots c_{q-2} \}/ab)$, where all the $c_i$ satisfy $R(a,b,c_i)$.

\item A $\mu'$ mapping $\bK^q_{0,V}$ into $Z$ is in\footnote{For
    simplicity, we write $\Uscr_{\tau'}$ to suppress the (uniform)
    dependence on the choice of $q$ and $V$.} $\Uscr_{\tau'}$ if it
    satisfies i) $\mu'(A'/B')\geq \delta_{\tau'}(B)$ and\footnote{Each of
    the $ |\aut(F_2)|$ of the renumberings the primitive extension that fix
    $ab$ yield isomorphic quasigroups since $V$ is $(2,q)$-variety. Setting
    $\mu'(\boldsymbol{\alpha}_q) = 1$ guarantees the result is a
    $q$-Steiner
   system.} ii) $\mu'(\boldsymbol{\alpha}_q) \geq 1$.

\item     Let $D' \in (\bK^{q}_{\mu',V},\leq')$ if and only if
    $\chi_{D'}(A'/B') \leq \mu'(A'/B')$.

    \end{enumerate}
\end{definition}

%

    Since both the restriction $\delta(A) \geq 0$ and the bound imposed
    by $\mu'$ are universally axiomatized it
is easy to check that $(\bK^{q}_{\mu',V},\leq')$ is smooth. However it is
$AE$-axiomatized because of clause \ref{K'mu}.2. Thus, the main difficulty in
proving Theorem~\ref{getsmquasigrp} is establishing amalgamation.

In \cite{BaldwinsmssII}, we gave a different construction which involves a
$\mu$ which counts good pairs in $\tau$ and a $\mu'$ which counts good pairs
in $\tau'$.  We write $\mu'$ here to emphasize that $\mu'$ counts good pairs
of $\tau'$-structures and for compatibility with the earlier notation. Unlike
\cite{BaldwinsmssII}, there is no dependence on a given $\mu$ defined on the
finite structures. 

\begin{theorem}\label{getsmquasigrp}
For each $q =  p^n$, each $\mu' \in \Uscr_{\tau'}$, and each Mikado-variety
of quasigroups $V$ with $|F_2(V) = q$,  there is a strongly minimal theory of
quasigroups, dubbed $T^q_{\mu',V}$, that interprets a strongly minimal
$q$-Steiner system.
\end{theorem}

\begin{proof} We now show the amalgamation for the $(\bK^{q}_{\mu',V},\leq')$, as in Lemma 5.11
and Lemma 5.15
of \cite{BaldwinPao}. Consider a triple $D, E, F$ in $\bK^{q}_{\mu',V}$ as in
Lemma~\ref{omitPasch}. That is, $D\subseteq F$ and $E$ is $0$-primitive over
$D'$. Since $E$ is primitive over $D$, although there may be a line contained
in the disjoint amalgam $G$ with two points in each of $D$ and $F-D$, each
line that contains 2 points in $E-D$ can contain at most one from $D$. If a
line contains three points from $D$, since $D$ satisfies
Definition~\ref{K'mu}.2 it is contained in $D$. Thus, there is no issue with
defining the relation $H$ on the disjoint amalgamation.  If $\mu'$ requires
some identification for some $(B,C)$, just as in \cite{BaldwinPao}, it is
because the (relational) $\tau'$-structure $BC$ is $DE$ and  there is a copy
of $C$ over $B$ in $F$ (Note the `further' in \cite[Lemma
5.10]{BaldwinPao}.).

The blocks of the Steiner system are the $2$-generated $*$-subalgebras. Now
the strong minimality of the generic follows exactly as in Lemmas 5.21 and
5.23
of \cite{BaldwinPao} and we have proved Theorem~\ref{getsmquasigrp}.
\end{proof}

For $q>3$, let $G_{H,R}$  be the group of automorphism for the countable
generic  of our construction  with vocabulary $\tau'$. $R$ is set-wise
invariant under the action of $G_{H,R}$ which is exactly the group of
$H$-automorphisms  but $H$ is not preserved by a permutation which setwise
stabilizes $R$.

We denote the theory of the generic $\Gscr_{\mu'} \myrestriction *$ by
$T_{\mu',V}$. We often drop the superscript $q$ as the specific $q$ is
irrelevant in further considerations.

\begin{remark}\label{noncomm} {\rm \cite[p5]{GanWer2} that depending on the choice
of the primitive $a$ in Definition~\ref{blalg}, the resulting $(r,k)$-algebra
may or may not be commutative. \cite[\S 5]{BaldwinVer} show that the strongly
minimal Steiner systems of \cite{BaldwinPao} have no non-trivial commutative
binary functions and deduce the theories do not admit elimination of
imaginaries.  Applying Theorem~\ref{getsmquasigrp} with the commutative
$(r,k)$ variety of block algebras yields a commutative strongly minimal
quasigroup.  Thus, more effort is needed to show it fails to eliminate
imaginaries.}
\end{remark}

\section{Strongly minimal block algebras, towers, and path graphs}\label{smba}
 \numberwithin{theorem}{subsection}

 The notion of an {\em (a,b)-cycle graph} is widely studied for finite
Steiner
{\em triple systems}. 
\cite{CameronWebb, Chicoetal} consider the notion for infinite Steiner {\em
triple} systems and prove the existence of infinite perfect and uniform
Steiner triple systems. {\em We generalize this notion to consider infinite
$q$-Steiner systems that are induced from strongly minimal
$(2,q)$-quasigroups} with $q$ a prime power \cite{BaldwinVer, BaldwinsmssII}.

We make the following assumption for this section. That is, we crystalize the
properties of the result of the construction in Section~\ref{consquasi} but
do not rely on any details of the actual construction.
 We will write $*$ to denote multiplication (as opposed to its graph $H$
 which was used to preserve the finiteness of structures in Section~\ref{consquasi} ).

\begin{assumption}\label{assT} $T$ is a strongly minimal theory in the
vocabulary\footnote{Here we write the function symbol $*$ rather than the
graph $H$ because we going to use the function to trace out a path.} $\tau' =
\langle *,R\rangle$ such that if $M\models T$
\begin{enumerate}
\item $(M,*)$ is a quasigroup in a Mikado variety $V$; \item $R$ is the
    graph of $*$.
\item There are functions $\delta$ and $d$ on the domain of $M$ that
    satisfy the properties of $\epsilon$ and $d$ in Section~\ref{hrucon};
    \end{enumerate}
    \end{assumption}

This assumption yields immediately that $(M,R)$ is a Steiner $q$-system where
$q = |F_2(V)|$. We use the $*$ operation to inductively construct a path on
points.

%
%

\subsection{Path Graphs}\label{pg}
\numberwithin{theorem}{subsection} \setcounter{theorem}{0}

Finite Steiner $3$-systems $Q$ are often studied via the cycle graph over
various $ab$; the pairs $(c,d)$ from $Q- \overline{ab}$ ($overline{ab}$ is
the line through $ab$) are colored red or blue depending on whether $a$ or
$b$ lies on the line $\overline{cd}$. Then a path is generated by choosing a
point $d$ off $\overline{ab}$ and starting with $\overline{ad}$ and
inductively choosing the line of a different color through the third point on
the current line. We extend this idea to $q$-Steiner systems.    It is
immediate that paths in Steiner $3$-systems do not intersect; so for strongly
minimal $3$-Steiner systems the definitions below reduce to those in
\cite{CameronWebb}. However, such disjointness is no longer immediate when
$q>3$ leading to the more complicated description of
    paths in Definitions~\ref{pathgraph2} and \ref{pathgraph3}.  In order
    to carry out the analysis, we exclude\footnote{This guarantees that the
    generator $d_1$ satisfies $d(d_1/\icl(ab)) \geq 0$.} from the graph,
    not just $\overline{ab}$ but the larger {\em finite} set $\icl(a,b)$, the
    smallest subset containing $a,b$ that is strong in $M$.  We later get
    stronger results by restricting the domain even further to $M-\acl(a,b)$.

\begin{definition}\label{pathgraph} Consider a Steiner system $(M,*,R)$
 determined by a   $q$-block algebra $(M,*)$ (Definition~\ref{rkdef}). For
    any $a,b \in M$, we will write $G_M(a,b)$ for the graph determined by
    the pair $a,b\in M$.

\begin{enumerate}
\item  The domain of $G_M(a,b)$ is $M-\icl(ab)$.

 \item For $x,y \not \in \icl(a,b)$, there is an edge colored    $a$
     (resp., $b$) joining $x$ to $y$  if and only if $R(a,x,y)$ (resp.,
         $R(b,x,y$)).

  \end{enumerate}
                \end{definition}

\begin{remark}\label{defch}
There is an edge coloured\footnote{Note that if $q=3$, this is
    the same as collinearity and we return to the framework of
    \cite{CameronWebb}.} $a$ (resp., $b$) joining $x$ to $y$ if and only if
    $a*x = y$ (resp., $b*x =y$).
    \end{remark}

 We have partitioned the lines ($R$-cliques) that intersect $\{a,b\}$ into
 $a$ and $b$ lines. Two lines with distinct colors can intersect in at most
 one point.


 We introduce certain  {\em paths} and then  in Section~\ref{indepcase}
{\em fans} in the graph that under appropriate hypotheses cover (most of)
the domain of the graph.
%




%

%

\begin{definition}\label{pathgraph2} Let $M \models T$ with $\mu \in
\Uscr^{\rm ls}$ (Definition~\ref{Kmu'}). Consider a $q$-block algebra
$(M,*)$ with associated path graph $G_M(a,b)$.
\begin{enumerate}
\item  For any $a,b$, we write $\overline{ab}$ to denote the  line of
    length $q$ generated by $\{a,b\}$.

 \item For $d_1 \not \in \icl(\overline{a b})$ we define  a sequence,
     denoted $P_{ab\dbar}$
     generated by $d_1 \in M-\icl(ab)$ over $\{a,b\}$ as follows.

     The   path $P_{ab\dbar}$ is the
sequence  $\dbar
=  d_1,  \ldots d_m$ such that $a*d_{2i+1} = d_{2i+2}$ and $b*d_{2i+2}
      =d_{2i+3}$ for $0\leq i \leq m$.

%
%
%

\item The {\em envelope}, $P^e_{ab\dbar}$, of the  {\em path},
    ${P}_{ab\dbar}$, with $\dbar = d_1 \ldots d_m$,is the union of
   the lines\footnote{We may sometimes write ${P}_{ab\dbar}$ when
    ${P}_{ab\dbar}- \{a,b\}$ is more precise; this is the usual
    ambiguity in describing good pairs $C/B$; technically $B$ and $C$
    are disjoint.} $\overline{d_i,d_{i+1}}$ for $1\leq i < m$. Note
    that if $i$ is odd (even), $a$ ($b$) is on
    $\overline{d_i,d_{i+1}}$.
\end{enumerate}
\end{definition}

    Note that if $e$ is on   an $a$-edge, $a*e$ is on the same line (and
similarly for $b$). Thus, the lines of the Steiner system are cliques of the
path graph.  But, if $e$ with $e\neq a$ and $e\not \in \icl(ab)$ is on an
$a$-edge multiplying $e$ by $b$ begins the generation of a distinct path,
$P_{ba\ebar}$  in the graph. We will show such a path is either an infinite
chain
or `cycles' by generating a $0$-primitive extension of $ab$.

\begin{definition}\label{pathgraph3}
\begin{enumerate}
    \item There are two possibilities when the process of
        Definition~\ref{pathgraph2} is iterated forward $m$ times.
        \begin{enumerate}
 \item An {\em $(a,b)$-chain} of length $m$ is a path $P_{ab\dbar}$
     with $\dbar
=d_1, d_2, \ldots d_m$ such that $a*d_{2i+1} = d_{2i+2}$ and $b*d_{2i+2}
 =d_{2i+3}$ for $0\leq i \leq m$

  and: for $j>i+1$ the lines $\overline{d_{i}d_{i+1}}$,
$\overline{d_{j}d_{j+1}}$ do not intersect. Thus $\delta(P_{ab\dbar})
=\delta(P^e_{ab\dbar})$. Note that $m$ counts the number of lines in the
path. We write $\boldsymbol{\sigma_m  }$ for the isomorphism type  of an
$m$-chain.  Note that, as in the 3-Steiner system case, the length of an
$m$-chain must be divisible by $4$.

\item At some stage the new line generated by $a, d_{2i+1}$ or $b,
    d_{2i+2}$ intersects one of the earlier lines in the envelope of the
    path. {\em In this case, we stop} the construction   with the new line.
    The result is an $m$-{\em pseudo-cycle}, an envelope
$P_{ab\dbar}$, such that for exactly one pair $(i,j)$ with
    $0\leq i \leq m$ and  $j>i+1$ the lines $\overline{d_{i}d_{i+1}}$,
    $\overline{d_{j}d_{j+1}}$ intersect.

     We write $\boldsymbol{\gamma_s }$ for an isomorphism type of an
    $s$-pseudo-cycle $P_{ab\dbar}$ 
     and
    $P^e_{ab\dbar}$ for the  isomorphism type of its

\item If the process continues infinitely we call the result an infinite
    chain.
    \end{enumerate}

\item Note that the construction of path through $d_1$ could equally well
    begin with the first line a $b$-line. In this case, we introduce a
    finicky notation. The $P_{bad_1}$ path\footnote{Switch $a$ and $b$ in
    the subscript.} through $d_1$ starts with a $b$-line.
%
%

\end{enumerate}
                \end{definition}

%
%

%
%



 Recall the construction stops as soon as there is a loop
but may be infinite.
In the pseudo-cycle case  $P^e_{ab\dbar}$ contains a minimal pseudo-cycle,
 which is
                $0$-primitive over $ab$. Thus, each triple $a,b$
                and $d \not \in \icl(ab)$, determine a unique mimimal path
$P_{abd}$ beginning with an $a$-edge; it may be a pseudocycle (perhaps
starting with a different $d'$) of minimal length or an infinite chain. {\em
While formally we have defined {\em pseudo-cycles} to emphasize the return
need be back  to the initial point, we will often write {\em cycle} for
short.}

Within the algebraic closure of $ab$ analysis by the graph structure is
    more complicated.  As, since any two points
     determine a line implies there are  $c \in \acl(ab)$ such that
      $d(c/ab)$ remains $0$ even when $c$ is an intersection point of many
      lines. Thus, in Section~\ref{primecase} we study inside  the graph over
      $\icl(ab)$ $\acl(ab)$ and in Section~\ref{indepcase} work over
      $\acl(ab)$.

%
\subsection{Inside $\acl(ab)$: Many Finite  paths}\label{primecase}

\numberwithin{theorem}{subsection} \setcounter{theorem}{0}

This subsection analyzes the structure of $G_M(a,b)$ when $M$ is a  prime
 model that  is algebraic over the empty set and for arbitrary $M$ the
 structure of  $\acl_M(ab)-\icl_M(ab)$. Section~\ref{indepcase} describes
 the properties of $(a,b)$-path graph off $\acl(ab)$.


While the definition in Section~\ref{pg} was primarily combinatorial (except
for the use of $\icl(a,b)$ rather than $\overline{ab}$, we now use the model
theoretic machinery about strongly minimal sets more heavily.

\begin{remark}[Towers]\label{towers}{\rm
Two prototypical properties of a strongly minimal theory $T$ are: a) the
existence of a unique {\em generic type} over the model whose restriction to
any set has infinitely many solutions and, as a result,  if $T$ has at least
two non-isomorphic countable models, b) the arrangement of the countable
models into a tower. Let $\langle M_j\colon 0\leq j< \omega+1\rangle$ be the
tower (elementary chain: $M_n \prec M_{n+1}$) of countable models of $T$,
with $M_0$ the prime model\footnote{The prime model of $T$ is the unique
model that can be elementarily embedded in each model.}; then $M_\omega$ is
isomorphic to the generic structure $\Gscr_{\mu,V}$ \cite[Lemma
5.29]{BaldwinPao}.  One might think  each $M_n$ is prime with an $\acl$-basis
of cardinality $n$;  we now show this is true when $\acl(\emptyset)$ is
infinite.
However, in Section~\ref{techs} we  provide choices of $T^q_{\mu',V}$ where
$M_0$ has dimension $2$ and so $M_n$ has dimension $n+2$.}


\end{remark}

%
%

%

 The cycles (using only partial lines of length three) played an
 important role in \cite{BaldwinPao}. We constructed the $2^{\aleph_0}$
distinct theories $T_\mu$ in \cite[Lemma 4.11]{BaldwinPao}, by showing (in
the vocabulary $\tau = \{R\}$) there were a countable family of $4n$-cycles
(actually back to the same element) that are mutually non-embeddible and
$0$-primitive over $2$-element sets. The choice of $\mu$ determines which of
these cycles are realized. Varying the argument slightly shows
as $s$ increases the $\boldsymbol{ \gamma}_s$  
(Definition~\ref{pathgraph2}.(1b)) induce infinitely many mutually
non-imbeddible primitives in $\bK_\mu$ over a two element set that is
strongly embedded. We also noted
in \cite[Lemma 4.11]{BaldwinPao} that there are infinitely many mutually
non-embeddible primitives in $\bK_\mu$ over the empty set and similarly over
a $1$-element set.

 $\Uscr^{\rm ls}$ allows $\mu$ that forbid the realization of specific good pairs
$B/\emptyset$. In \cite{BaldwinPao}, we showed the algebraic closure of the
empty set was infinite if the generic contained a copy of the Fano plane, --
the unique $7$-element projective plane, $F$. So setting $\Fscr$ as the
collection of $\mu \in \Uscr$ with $\mu(F/\emptyset) >0$ guarantees
$\acl_M(\emptyset)$ is infinite for any $M\models T_\mu$.  We retain the name
$\Fscr$ but make it a much larger subset of $\Uscr^{\rm ls}$.

\begin{notation}\label{fanoset}
 Let
 $\Fscr$ be the set of $\mu'\in \Uscr^{\rm ls}  $ such
 that $\mu(C/\emptyset) > 0$ for some good pair $C/\emptyset$.
\end{notation}

\begin{lemma}\label{infacl} If $\mu' \in \Fscr$ and $M \models T_{\mu',V}$ then
$\acl_M(\emptyset)$ is infinite.
\end{lemma}

\begin{proof}
It is easy to see that any $C$ that is $0$-primitive over $\emptyset$ must
contain two intersecting lines so three non-collinear points exist.  Noting
that the only use  in Lemma 5.27 of \cite{BaldwinPao} of the assumption that
the Fano plane is imbedded in $M$ is to guarantee that there are three
non-collinear point in a subset of $M$ that is $0$-primitive over
$\emptyset$, we get an infinite algebraic closure here.  The construction of
an infinite tower of $0$-primitive extensions uses only that $\mu \in
\Uscr^{\rm ls}$.
\end{proof}

  In Section~\ref{infcyc}, we give several examples of strongly minimal
quasigroups where the dimension of the prime model is $2$.


\begin{lemma}\label{primeinfcyc} If $M \models T$ with
$ \acl_{M_0}\neq (\emptyset)$ there are infinitely
many disjoint (over the finite $\icl_N(ab)$) finite cycles in $G_{N}(a,b)$,
where $N$ is a copy of the prime model of $T_\mu$ with $N \supseteq \{a,b\}$.

\end{lemma}

\begin{proof} Such an $N$ exists by Lemma~\ref{infacl}. Fix $\Dscr$ as
$\icl_M(a,b)$. For each $i$ there is a pseudocycle $C_i$ that is a primitive
extension over $\icl(ab)$ based on $ab$ with length $4i$. The structure
 with domain   $\Dscr \cup C_i$ is denoted $\Ascr_i$.
Since $\mu(C_i/ab) \geq \delta(ab) =2$, there is an embedding of $\Ascr_i$
into the saturated (also generic) model $M_\omega$. But $N \prec M_\omega$
and is algebraically closed so the image of $C_i$ is in $N$. Now, since the
$C_i$ are $0$-primitive over $\icl_N(ab)\leq N$, the $\Ascr_i$ are disjoint
over $\Dscr$.

\end{proof}

\begin{question} Can the prime model contain an infinite chain?  Is there any
decomposition by chains of the prime model? Compare these questions with the
alternative decomposition of the prime model by taking the union of tree
decomposition by normal subsets in \cite{BaldwinVer}.
\end{question}

\begin{question}\label{allcycfin}
By using  the more radical alterations of the construction as in
Section~\ref{infcyc}, can we have all cycles in the prime model finite by
insisting exactly one
    isomorphism  type of a pseudocycle is consistent, say, a 4-pseudocycle?
\end{question}

\subsection{Over $\acl(ab)$ all paths are infinite}\label{indepcase}

\numberwithin{theorem}{subsection} \setcounter{theorem}{0}

We study those paths in $G_M(a,b)$ that are generated by $d_1\not \in
\acl_M(a,b)$.  We justify in Lemma~\ref{disjointness} the following notation:

%

\begin{notation}\label{pathnote} For $d_1\not \in \acl_M(a,b)$,
 $P_{abd_1}$  ($P^e_{abd_1}$)  denotes the (envelope of)
the longest path generated by  beginning with $ad_1$. This path may be
infinite.
\end{notation}

\begin{lemma}\label{disjointness} Suppose $d_1 \not \in \acl(a,b)$.

\begin{enumerate}
\item $d(d_1/ab) =1$; the path generated by $d_1$ is infinite.
\item  Distinct $a$-edges in the path  $P_{ab\dbar}$   cannot intersect;
    but each $a$-edge intersects $q-1$ $b$-edges.

\item If $P_{ab\dbar}$ is an infinite path then for every $X \subseteq P^e_{ab\dbar}$,
$d(X/ab) =1$.
\item If $P_{ab\dbar}$ an infinite path there is exactly one $e$ on
    $P^e_{ab\dbar}$ that is on an $a$-line and $P_{bae}$ is an infinite
    path (Recall Definition~\ref{pathgraph3}.2).
\end{enumerate}
\end{lemma}
\begin{proof}  1)
If $d_1 \in M - \acl(a,b)$,  $d(d_1/ab) =1$; otherwise $d_1\in \acl(ab)$. If
$P_{ab\dbar}$ is finite, it is because some $C \subseteq P_{ab\dbar}$ is
 a pseudocycle. But the $\delta(C/ab) = 0$ and $d_1 \in \acl(ab)$.

2) If $\overline{ad_{2i}}$ is a line in $P_{ab\dbar}$ then for
 any element $x \in  \overline{ad_{2i}}$, $a*x \in \overline{ad_{2i}}$. But for each
of the $q-1$ non-trivial  star
 terms, $t(x,y)$,
 $b * (t(a,d_1)$ generates a new line.

3) 
Suppose (without loss) that $d_1 \subseteq X \subseteq P^e_{ab\dbar}$ and
$d(X/\acl(ab)) =0$. Then $d_1 \in \acl(ab)$. Two paths generated by distinct
$d_i \not \in \acl(ab)$ can intersect in one point; $d(P_{abd_1} \cup
P_{abd'_0})= 1$. But if there are two points of intersection   $d_1 \in
\acl(ab)$.

4) For any such $e$ there is a line determined by $b,b*e$. But this line
generates an infinite path only if $d(e/ab) = 1$. Now apply 2).
\end{proof}

With these results in hand we see that actually $a,b, d_1$ generate a fan of
lines.

\begin{definition}\label{fan}
The {\em fan} generated by $abd_1$ is defined by induction.
\begin{enumerate}
\item $F^0_{abd_1}$ consists of all points on envelopes of paths generated
    by a line  $ae$  where  $e$ is on a $b$ edge of $P_{abd_1}$  or by a
    line  $be$  with  $e$ on an $a$-edge of $P^e_{bad_1}$;

\item $F^{n+1}_{abd_1}$ consists of all points on envelopes of paths
    generated by lines  $ae$   where  $e$ is on a $b$ edge of $P^n_{abd_1}$
    or by a line  $be$  with  $e$  on an $a$-edge of $P{^e}^n_{bad_1}$;

%
\item The fan $F_{abd_1} = \bigcup_{n< \omega} F^n_{abd_1}$.
\end{enumerate}
                \end{definition}

                Note that
                $F_{abe} =F_{baf}$ if $e$ and $f$ are both on the same line in $M-\acl(a,b)$
                through $a$ (or through $b$).

                As in  Lemma~\ref{disjointness}, we see  immediately that if two fans intersect in
                a single point their union is
                a larger (not definable) subset of rank $1$:

                \begin{lemma}\label{fanuni}
                Two fans can intersect
in at most one point.
\end{lemma}

\begin{theorem}\label{fancover} If $M$ is countable and $\dim(M/N) =1$, then for any $a,b \in M$, $M$ is a union of fans
over $N$.  Inductively, the conclusion applies to any $M'\succ M$.
\end{theorem}

\begin{proof} Let $\langle e_i: i<\omega\rangle$ enumerate $N-M$. Fix any $a,b \in N$,
 choose $e_0= d_1 \in M-N$ and let $F_0$ be the fan $F_{abd_1}$.
Now for each $n$, let $d_{n+1}$ be $e_j$ for the least $j$ such that $e_j
\not \in N \cup F_n$.  Clearly $\bigcup_{n< \omega} F_n \cup N = M$. Since
the dimension $N/M$ is $1$, there will be algebraic relations among the fans.
However, any two can intersect in at most one point and by construction there
graph edges ($a$ or $b$ lines) that are not in one of the listed fans.
However, many instance of $R$ are not in the graph.
%
%
\end{proof}

\subsection{No Perfect Path graphs}\label{noperf}
\numberwithin{theorem}{subsection} \setcounter{theorem}{0}

Cameron and Webb \cite{CameronWebb} extend to infinite structures the notion
of a perfect Steiner triple system as one in which each cycle graph $G(a, b)$
is a single cycle.  They find $2^{\aleph_0}$ countable such Steiner triple
systems.  In line with Definition~\ref{pathgraph},
 we can extend this definition to any $q$-block algebra.   However,
  we show none of the $q$-Steiner systems satisfying a $T$ obeying Assumption~\ref{assT}
  are perfect. Clearly there can be no uncountable perfect Steiner $k$
system in any
 reasonable sense since whatever replaces `cycle' will be countable. We will
 take the weakest plausible notion, which includes a single path or a fan;
 we show no such complex  covers
 $M-\acl({ab})$, when $M \models T$.  In Theorem~\ref{fancover},
 we covered $M-N$ by at most $|M-N|$
 fans, but not finitely many.

\begin{definition}[Perfect]\label{defperf} If $(M,*,R) \models T$
we say  $(M,*,R)$ is a {\em
perfect $q$-Steiner system} if for some finitely generated $R$-closed set
(Definition~\ref{Rcl}) $X = M-\acl({ab})$.
\end{definition}

Since every line in a Steiner system associated with a $q$-Steiner system is
two-generated as a quasigroup, we can think of
$R$-closure as finding the generated sub-quasigroup.  
Omer Mermelstein suggested the key idea for the  proof for the following
result.

\begin{lemma}\label{Rdiminf} If $M$ is a model of $T$,
  $A\leq M$, and  $|M-A|$ is infinite,
then $M$ has infinite $R$-dimension.
\end{lemma}

\begin{proof} We first show that if $C$ is $0$-primitive  over $A$ and $a \not \in
C\cup A$, $A^*$, the $R$-closure of $Aa$, does not intersect $C$.  Note by
induction that
every finite $E \subseteq (A^*-A) $ satisfies $\delta(E/A) = 0$. Now, fix
an enumeration $A^*$ such that $e_j \in \cl_R(\{e_i:i<j\} = E_j$. Suppose
for contradiction $A^* \cap C \neq \emptyset$ and choose the least $k$ with
$e_k \in C \cap A^*$. But then $e_k$ witnesses  an edge between $C$ and
$E_{j+1}$; this implies $\delta((E_j \cup C)/A) < 0$, contrary to
hypothesis.

There  are infinitely many incomparable $0$-primitives $C_j$ over $A$
(\cite[Lemma 4.11]{BaldwinPao}; choose successively, a seed $a_j$ in each
$C_j$. Applying the first paragraph, we see the $\cl_R(Aa_j)$ are mutually
disjoint. By constructing $\langle A_j,A^*_j\rangle$ by the procedure of the
last paragraph, we witness infinite $R$-dimension.
\end{proof}

%
%
%

\medskip
Since a perfect Steiner system is the $R$-closure of finitely many elements,
we have immediately from Lemma~\ref{Rdiminf}:

\begin{corollary}\label{noperfthm}
If $(M,*,R) \models T$ satisfies Assumption~\ref{assT},
$(M,*,R)$ is not a Steiner perfect system.
\end{corollary}

%
%
%

\begin{question} In \cite{BaldwinVer}, we show the definable closure of a
strongly minimal system $(M,R)$ is essentially unary if $T_\mu$ is triplable
(For any primitive $C/B)$, $\mu(C/B) \geq 3$.). In the expanded vocabulary
 $\tau'$, models of $T_{\mu',V}$ have $*$ as a non-trivial binary function. But,
assuming $\mu'$ is triplable, are there any binary functions that are not
polynomials in $*$?
\end{question}

\section{Uniform Path graphs and $2$-transitive structures }\label{techs}

In Section~\ref{smba}, we studied theories  $T$ which satisfied the
properties of $T^q_{\mu',V}$ of quasigroups built by a Hrushovski style
construction as in Section~\ref{consquasi} where $\mu' \in \Uscr_{\tau'}$ and
for any $M\models T_{\mu',V}$, $\acl_M(\emptyset) \neq \emptyset$.  Unlike
the previous  section, we now make major modifications to the construction to
consider subsets where algebraic closure has few pseudo-cycles and to find
$2$-transitive structures.  Thus, we return to the complicated notation
$T^q_{\mu',V}$ to clarify where the construction is changing.

In Section~\ref{indepcase} we found examples where all cycles were infinite
 when we took the domain of the path graph as $M-\acl(ab)$. But in
 Section~\ref{primecase} with domain $M_0-\icl(ab)$  we always had finite
 cycles and the existence of infinite cycles in the prime model is an open problem. In this
 section we restrict our attention to the domain, $M-\icl(ab)$. We first
 (Section~\ref{infcyc}) modify the construction to be able to specify
 which, if any, finite cycles occur. In Section~\ref{unif} we introduce the
 notion of a uniform (The isomorphism type of $G_M(a,b)$ does not depend
 on the choice of $a,b$.) $q$-Steiner system (generalizing \cite{CameronWebb,
 Chicoetal}). Then by different methods in Sections \ref{uniformity} and
 \ref{finfincyc} we construct families of  $2$-transitive and hence uniform $q$-Steiner
  systems.

%
%
%

We use two model theoretic methods    to solve some problems suggested from
the study of cycle graphs in \cite{CameronWebb}. These methods modify the
theory $T_{\mu',V}$ either by changing $\mu$ or, more drastically,
restricting the class $\bK_0$ of finite structures.  And then we combine
the two in Section~\ref{finfincyc}.

\subsection{ All paths are infinite }

\label{infcyc}

\numberwithin{theorem}{subsection} \setcounter{theorem}{0}

 In this section, we find $T_{\mu'',V}$ whose models have no finite cycles.
  It is then easy
 to allow certain specified finite lengths of cycles.
%
%
The key point here is to vary the class $\Uscr_{\tau'}$ from
Definition~\ref{defdelmu}.4 maintaining the amalgamation  so the resulting
generic model is strongly minimal but preventing finite cycles. As in
Section~\ref{consquasi}, we work in a vocabulary $\{H,R\}$, where $R$ is
collinearity in a linear space and model $H$ is the graph of a quasigroup
operation $*$.  We introduce a set $\Bscr$  of $\mu''$ obtained by
 modifying $\mu'\in \Uscr_{\tau'}'$ to $\mu''$ by changing the value {\em only on}
  the isomorphism types good pairs $C/\{a,b\}$ which are pseudo-cycles.
  As $\Bscr $ and $\Uscr_{\tau'}$ differ on pseudocycles, apparent  contradictions between here and Section~\ref{smba}
  are resolved.

  \begin{definition}\label{defbscr}
 Recall from Definition~\ref{K'mu} that $\boldsymbol{\gamma_n}$ denotes an
  isomorphism type of a pseudo-cycle over a two element set. Let $\Bscr$
  denote the set of $\mu''$ obtained by for every $n$, redefining each
   $\mu'\in \Uscr_\tau$ to $\mu''$ by setting
 $\mu''(\boldsymbol{\gamma_n}) = 0$ for each $n$.

\end{definition}


We define a class $\bK'_{\mu'',V}$ whose generic has only infinite
cycles. Thus there are no finite cycles in any model of $T_{\mu'',V}$.

\begin{lemma}\label{apmu2}
If $\mu'' \in \Bscr$, for each $q,V$,the class of $\tau' =
\{*,R\}$-structures $\bK^{q}_{\mu'',V}$ from Definition~\ref{K'mu} has the
$\leq$-amalgamation property.   If $\mu''\in \Bscr$, every model of
$T_{\mu'',V}$ has only infinite cycles.
\end{lemma}
\begin{proof}
We must check that we can complete the amalgamation while insisting that
for each $n$, $\boldsymbol{ \gamma_n}$ is omitted.  For this we must
slightly vary the proof of Lemma 5.10
in \cite{BaldwinPao}, whose notation we follow. Let $F, E \in
\bK^{q}_{\mu'',V}$. Now, let $G = E \oplus_{{D}} F$, where $(D,E)$ is a good
 pair (with $|E-D|
>1$) and $((a,b),C_k)$ is a good pair witnessing $\boldsymbol{\gamma}_k$ (So $C_k$ is a
pseudo-cycle.). The difficulty is that the good pair $(C_k/B)$  does not
satisfy the requirement $\mu(C_k/B) \geq \delta_{\tau'}(B)$.  We gave a
separate argument to show no $\boldsymbol{\gamma}_k$   blocks amalgamation;
the result then follows without change. There are no realizations of the
good pair $\boldsymbol{\gamma}$ in any of $D,E,F$; we must show it is not
 realized in $G$. The crux is that, by definition of $((a,b), C_k)$, for
 any $k,i$,
  each $c_i\in C_k$ is on a separate triple in $R$ with each of $a$ and
 $b$.  Now if $(a,b) \subseteq F$ (compare Case B.1 of \cite{BaldwinPao}),
 each $C_i$ must be contained in $F$ or else there is a clique
 ($ac_ic_{i+1})$, modulo renaming, with two elements in $F$ and one in
 $E-F$ contradicting the primitivity of $E$ over $D$. If  one of $a,b$, say
  $a$ is in $E-F$  then for each $i$, $C_i \subseteq E$ or the line between
  $a$ and $c_i$ is based in $D$ (Definition 3.11 of \cite{BaldwinPao}) and
  that is clearly impossible, since it contradicts that $E$ is primitive
  over $D$; so each $C_i \subseteq E$. But now, since $E$ doesn't realize
  $\gamma_n$, $b$ must be in $F-D$ and $C_i \cap (E-D) \neq \emptyset$; we
  get the same contradiction. So $C_i \subseteq D$.   But now $a \in E-D$
  is on a line based on $C_i \subseteq D$, contradicting the primitivity of
  $E$ over $D$. Thus for any $M \models T_{\mu'',V}$, $a,b \in M$ and $d_1 \not \in
  \icl(ab)$, $P_{abd_1}$ is infinite. So we finish.
  \end{proof}

A simple variant on the argument for Corollary 5.3 of \cite{BaldwinPao}
(Replace `for every $n$' in Definition~\ref{defbscr} by `for $n\in X^c$'.)
shows we can  omit arbitrary sets of $\boldsymbol{\gamma_n}$:

  \begin{theorem}\label{BPuniform}
For any $X \subseteq \omega$ of numbers divisible by $4$   and $\mu \in
\Uscr$, we can construct still another variant $\mu^X$ of $\mu$ such that
models of $T^q_{\mu^X,V}$ realize an $n$-pseudo-cycle if and only if $n\in
X$.
\end{theorem}
One cannot simply modify $\Uscr$ to say  all points have  trivial algebraic closure
and carry out the amalgamation argument.
Omer Mermelstein provided the following counterexample, showing some
 restriction, such as to the $\gamma_n$, is necessary for
 Lemma~\ref{apmu2}.
Here is an amalgamation diagram where the good pair $C/B$ does not appear in
any of the
  components but is in the amalgam. Nevertheless, we give several examples in later sections
  where $\acl_{M_0}(\emptyset) = \emptyset$.


  \begin{example}{\rm Let $B$ consist of five points $a,b_1, \ldots b_4$ and $C$ consist of four points
  $c_1, \ldots c_4$, where $R(c_i,b_i,c_{i+1})$  for $i = 1, \ldots , 3$,
  $R(a,b_2,b_3)$, and $R(c_4,c_1,b_4).$ Then $C$ is $0$-primitive over $B$.
  But now if we let $D_1 = \{a,c_2,c_4\}$,  $D_1 = \{b_1,c_1, b_4\}$ and
  $D_2 = \{b_2,c_3, b_3\}$ we have $D_1 \leq D_1$ and $D_1 \leq D_2$, but
  $BC$ appears in the amalgam.
  }
  \end{example}
%

\subsection{Uniform $G(a,b)$}\label{unif}
\numberwithin{theorem}{subsection} \setcounter{theorem}{0}

\cite{CameronWebb} call a Steiner system uniform if all the cycle graphs
$G_M(a,b)$ are isomorphic.  \cite{Chicoetal} construct  $2^{\aleph_0}$
countable uniform sparse infinite Steiner triple systems.  We obtain
$2^{\aleph_0}$ families of countable uniform infinite Steiner systems for
each prime power $q$.

%

We adapt the Cameron-Webb notions of uniform
 \cite{CameronWebb} to accommodate $q$-Steiner systems.  Recall
 (Definition~\ref{pathgraph}) that the domain of $G_M(a,b)$ is $M-\icl(a,b)$.
We will consider cases where $\acl(a,b)$ is both finite and infinite.

\begin{definition}[Uniform]\label{defuni} We say a model $(M,*,R)$ of $T^q_{\mu',V}$ is {\em uniform}, if for
any $(a,b)$, $(a',b')$, $G_M(a,b) \simeq G_M(a',b')$.
\end{definition}

Here is a sufficient condition for uniformity.

\begin{lemma}\label{gettrans}
\begin{enumerate}\item If $(M,*,R)$ is  a model of a theory $T$ generated by a Hrushovski
class (Definition~\ref{hruclass}) of linear spaces
%
such that every two element set $A$ satisfies $A \leq M$, the
automorphism group of $(M,*,R)$ acts $2$-transitively on $(M,R)$. \item
Clearly, if the automorphism group of $(M,*,R)$ acts $2$-transitively on
$(M,*,R)$, $(M,*,R)$ is uniform.
\end{enumerate}
\end{lemma}

\begin{proof} Since all pairs $(a,b)$  are isomorphic and each is embedded strongly
in the generic $\Gscr$, the result is immediate for $\Gscr$.  But this
transitivity  extends to all models since if one model of a complete theory
has a single $2$-type, all models do. And, each model of a strongly minimal
theory is finitely first order homogeneous (finite sequences realizing the
same first order type are automorphic) (e.g.\! \cite[Theorem
5]{BaldwinLachlansm}).
\end{proof}

\subsection{$2$-transitive $M$, $3$-Steiner systems, Changing $\bK_0$
}\label{uniformity}

 \numberwithin{theorem}{subsection} \setcounter{theorem}{0} In
 Section~\ref{infcyc} we showed  that, by modifying the set of possible
 $\mu$, we could ensure that there were no finite pseudo-cycles. The
Steiner system in Section~\ref{infcyc} was far from uniform as there were
many 2-types, e.g. pairs with  non-isomorphic algebraic closures. (We only
restricted those primitive extensions that were pseudo-cycles.)


We have dealt with two variants of the Hrushovski construction. Recall that
in the linear space case we used $\bK_0$ to play the role of $\bL_0$ in
Notation~\ref{hruclass}.  We constructed generics in both $\tau$ and $\tau'$,
with the same basic construction. But in the more general context of
Definition~\ref{hruclass} we can restrict $\bK_0$ before beginning the
construction
and realize the hypothesis of the general statement of
Lemma~\ref{gettrans}.1.




 In Section 5.2 of \cite{Hrustrongmin}, Hrushovski proves there are
$2^{\aleph_0}$ strongly minimal $\tau$-structures with pairwise
non-isomorphic associated combinatorial geometries. He achieves this by
ensuring that algebraic dependence of a triple $a,b,c$ is equivalent to
$R(a,b,c)$. Mermelstein pointed out to me that these structures are in fact
Steiner triple systems. We will see that they are $2$-transitive and every
cycle is infinite.  Example~\ref{exhru} is considerably more restrictive than
the linear space examples; it not only forces that two points
 determine a line  but also that every full line
  has 3 points.  In Theorem~\ref{restrictk0} we show less drastic surgery
  on the \cite{BaldwinPao} construction still allows us to find uniform
  $G(A,B)$-graphs when $q >3$.

\begin{example}\label{exhru} \cite[Example 5.2]{Hrustrongmin} 
{\rm We denote the theories described in  this example by $T_{H, \mu}$.
 The dimension function $\delta_H$ is the usual: $\delta_H(A) = |A| - |R|$,
where $|R|$ is the number of $3$-element subsets of $A$ satisfying $R$ and
strong submodel is defined in usual way. The novelty was in use of the
$\delta$-condition to define $\bK^H_0$. Namely, the collection  of finite
structures $C$ such that every subset $B$ of $C \in \bK^H_0$  with power at
most $3$ is strong in $C$:

$$(*) \hspace{3pt} \bK^H_0 = \{A\colon B\subseteq A
 \wedge |B| \leq 3 \rightarrow B\leq  A\}.$$

Since the  amalgamation of Hrushovski's basic example added no edges, this
subclass also has amalgamation by the same amalgam.
For each $\mu$, $\bK_{H,\mu}$ is to $\bK^H_0$ as $\bK_\mu$ is to $\bK_0$
(Definition~\ref{defdelrank}).

We obtain a linear space by interpreting $R$ as collinearity. Two points
determine a line as $R(a,b,c)\wedge R(a,b,c) \wedge\neg R(a,b,d)$ makes
$\delta(\{a,b,c,d\}) =2 < \delta(\{b,c,d\})$. Since any non-trivial
$0$-primitive over a two element set contains $3$ non-collinear points, (*)
implies the algebraic closure of two points is the third point on the line
they determine. Thus there are two quantifier-free configuration of three
points: dependent, independent. Since, by $(*)$, both configurations are
strong in the generic, they determine by homogeneity, as in
Lemma~\ref{gettrans}, the two possible $3$-types. Similarly property $(*)$ of
this Hrushovski example makes it a Steiner triple system\footnote{This
example will not permit lines with longer length by modifying $\mu$. As,
there can be no $4$-clique, $\ell$, since with the Hrushovki definition
$\delta(\ell) =0$ while $\delta$ of two points is $2$.}.}
\end{example}




%

Here we write cycle since we are dealing with a Steiner-triple cycle and no
path can be a proper pseudo-cycle as opposed to a cycle.

\begin{fact}\label{Huniform} For any $\mu$ and any $(M,R) \models T_{H,\mu}$,
$(M,R)$ is a strongly minimal uniform Steiner triple system. In fact, the
algebraic closure of any pair is the third point on the line through $a,b$
and so each cycle is infinite.
\end{fact}

\begin{proof} As noted in the description of Example~\ref{exhru}, in $(M,R)$
the algebraic closure of a pair is the line through them.  Since there are
only two $3$-types of tuples extending $(a,b)$, any two $d_i$ that are not on
the line $ab$ are isomorphic over $a,b$ and thus the cycles they generate are
isomorphic. The last claim is immediate since all points not on the line are
automorphic over $ab$. Since any potential finite pseudo-cycle over $a,b$ is
in $\acl(ab) = \{a,b,c\}$, where $R(a,b,c)$, there are no finite
pseudo-cycles.
\end{proof}
%




\subsection{$2$-transitive $q$-Steiner systems; Changing $\bK_0$ and $\bU$}\label{finfincyc}
\numberwithin{theorem}{subsection} \setcounter{theorem}{0}

 We turn to a different method\footnote{This approach of restricting
primitives over very small sets
to establish various amounts of transitivity of the non-Desguaresian plane
appears in \cite{Hrustrongmin,Baldwinautpp}.}
 to obtain  uniformity results for Steiner $q$-systems for any prime power
$q\geq 3$ and to restrict the number of  finite cycles. We combine  a variant
of the Hrushovki's Example~\ref{exhru} with modifying $\mu$ to control
  a second fundamental
 invariant: number of cycles.

  \begin{definition}\label{defKB}
   We write $\bK^J_0$ for the class of linear spaces such that
$$ (**) \ |B|\leq 2 \  \text{implies}\  B\leq A$$ for every finite linear space $A\in \bK^2_0$  containing $B$.
We write $\bK^J_{\mu'',V}$ for the class determined by $**$, $\mu'' \in
\Bscr$  (Definition~\ref{defbscr}) and , a Mikado variety of quasigroups $V$.
%
\end{definition}

%
%

%
%

As in Example~\ref{exhru}, $(**)$ and Lemma~\ref{gettrans} imply every two
element set is strong, so each model is $2$-transitive. There are two
differences from Example~\ref{exhru}: i) the strong substructure notion is
with respect to the $\delta$ in \cite{BaldwinPao} and so we can vary the line
length; ii) we don't kill  the entire (non-trivial) algebraic closure of each
2-element set but explicitly forbid only the finite cycles. We note below
that we can allow finitely many cycles over each pair $(a,b)$.

\begin{theorem}\label{restrictk0}
If $\mu\in \Bscr$ (Definition~\ref{defbscr}), $$\bK^J_{\mu'',V}$$ has
amalgamation,
 the generic (and hence every  model)  has no finite
 paths and is $2$-transitive so the path graph is uniform.
\end{theorem}

\begin{proof}
The amalgamation follows  {\em mutatis mutandis} from Lemma~\ref{apmu2}.
Note that $(**)$ implies every two element set is strong, so each model is
$2$-transitive. This holds in every model by Lemma~\ref{gettrans}; hence
$G_M(a,b)$ is uniform. Finite paths are blocked, since $\mu \in \Bscr$.
\end{proof}

As we modified Lemma~\ref{BPuniform}, we modify the proof of
Theorem~\ref{restrictk0} to get:


\begin{theorem}\label{BPuniform1} If $\mu'' \in \Bscr$  then for any variety $V$
and for any model $(M,*,R)$ of $T^q_{J,\mu'',V}$ and any $(a,b)$, both
$\acl_M(\emptyset) = \emptyset$ and $(M,*,R)$ is uniform.

Further, for any finite set  $X$ of pairs, $(n_i,m_i)$  with $n_i$ divisible
by $4$, we can construct a theory $T^J_X$ such that if $(M,*,R) \models T_X$
and $(a,b) \in M$, $G_m(a,b)$ has $m_i$ cycles of length $n_i$.

\end{theorem}

\section{Questions}\label{questions}
We close by suggesting some more traditional combinatorial questions
suggested by the examples here.

\begin{question}\label{qg1}{\em We have studied path graphs in strongly minimal $q$-Steiner systems
 induced by quasigroups.
But our definition has no reliance on strong minimality, although our
arguments do.} What can be learned by more traditional combinatorial methods
about the structure of path graphs in arbitrary {\em finite or infinite}
$q$-Steiner systems
 induced by quasigroups?
 \end{question}

 \begin{question}\label{qg2} {\rm We built strongly minimal quasigroup that induce
 $q$-Steiner systems using the $\mu$-function only with respect to collinearity.
 Suppose one moves closer to
 the setting of \cite{HorsleyWebb}. }  Can one construct   an infinite quasigroup
 by considering finite quasigroups from a Mikado variety
 (Definition~\ref{rkdef}) while omitting specified finite configurations as
 in  \cite{HorsleyWebb}? It seems each case would require its own variant on
 amalgamation.
 Is there a way to recover the local
finiteness of the generic as in \cite{BarbinaCasa}? If so, what is the model
theoretic
 complexity of the resulting theory?
 \end{question}

%
%



The last question depends on understanding the Lenz-Barloti classification.

\begin{question}\label{qg3} {\em In \cite{Baldwinasmpp} (using the methods of Section~\ref{finfincyc})
  a Morley rank $2$
$\aleph_1$-categorical non-desarguesian projective planes  is coordinatized
by a ternary ring that is not linear. The non-linearity means that while the
quasi-groups for both addition and multiplication are definable, they cannot
be composed to give the ternary $t(x,y,z) = xy + z$ that arises in a division
ring.  That is, the plane is at the lowest level in the Lenz-Barlotti
hierarchy.} Could similar but less radical surgery  yield
$\aleph_1$-categorical non-desarguesian projective planes that are higher in
that hierarchy? \end{question}


\end{document}